\documentclass[12pt]{article}

\usepackage{amsmath}
\usepackage{amsmath, amsthm, amsfonts, amssymb, color}
\usepackage{mathrsfs}
\usepackage{latexsym,bm}
\usepackage{graphicx}

\theoremstyle{definition}
\newtheorem*{ack*}{Acknowledgment}
\theoremstyle{remark}

\newtheorem{thm}{Theorem}[section]
\newtheorem{asp:mean}{Assumption}[section]
\newtheorem{lem}[thm]{Lemma}
\newtheorem{prop}[thm]{Proposition}
\newtheorem{cor}[thm]{Corollaray}

\newtheorem{remark}{Remark}[section]

\newtheorem{exa}[thm]{Example}

\newcommand{\p}{\mathrm{P}}
\newcommand{\pp}{\mathbb{P}}

\newcommand{\R}{\mathbb{R}}
\newcommand{\e}{\mathrm{e}}

\definecolor{wco}{rgb}{0.5,0.2,0.3}

\numberwithin{equation}{section} 

\begin{document}

\allowdisplaybreaks

\title{\bf Stationary measures and the continuous-state branching process conditioned on extinction}
\author{ \bf Rongli Liu, Yan-Xia Ren,
Ting Yang }
\date{}
\maketitle

\begin{abstract}
We consider continuous-state branching processes (CB processes) which become extinct almost surely. First, we tackle the problem of describing the stationary measures on $(0,+\infty)$ for such CB processes. We give a representation of the stationary measure in terms of scale functions of related L\'{e}vy processes. Then we prove that the stationary measure
can be obtained from the vague limit of the potential measure, and, in the critical case, can also be obtained from the vague limit of a normalized transition probability. Next, we prove some limit theorems for
the CB process conditioned on extinction in a near future and on extinction at a fixed time. We obtain non-degenerate limit distributions which are of the size-biased type of the stationary measure in the critical case and of the Yaglom's distribution in the subcritical case. Finally we explore some further properties of the limit distributions.
\end{abstract}

\bigskip

\noindent\textbf{Mathematics Subject Classification 2020:} Primary 60J80; Secondary 60F05

\bigskip

\noindent\textbf{Keywords and Phrases:} continuous-state branching process, stationary measure,
vague convergence, conditional limit theorems, size-biased  measure.

\section{Introduction} \label{sec:I}
A $[0,+\infty)$-valued strong Markov process $Z=(Z_t)_{t\ge 0}$
, with probabilities $\{\pp_{x}: x\ge 0\}$ and absorbing state $0$,
is called a continuous-state branching process (CB process for short) if it has paths which
are right continuous with left limits, $\pp_x\left(Z_0=x\right)=1$ for every
$x\ge 0$, and it employs the following
branching property: for any $\lambda\ge 0$ and
$x,y\ge 0$,
\begin{equation*}
\pp_{x+y}\left[\e^{-\lambda Z_t}\right]=\pp_x\left[\e^{-\lambda Z_t}\right]\pp_y\left[\e^{-\lambda Z_t}\right],
\end{equation*}
where $\pp_x$ denotes the expectation with respect to the probability $\pp_x$.
We suppose that $Z$ has branching mechanism $\psi$,
which is specified by L\'evy-Khintchine formula
\begin{equation}\label{def:psi}
\psi(\lambda)=\alpha\lambda+\frac{1}{2}\sigma^2\lambda^2+\int_0^{+\infty} (\e^{-\lambda r}-1+\lambda r)\pi(\mathrm{d}r),\quad \lambda\geq 0,
\end{equation}
where $\alpha\in\mathbb R,\sigma\geq 0$, and $\pi$ is a positive Radon measure on $(0,+\infty)$ such that
    $\int_0^{+\infty} (r\wedge r^2)\pi(\mathrm{d}r)<+\infty$.
One has $\pp_x[Z_t]=x\e^{-\psi'(0+)t}$ for all $x,t\ge 0$.
Since $\psi'(0+)=\alpha$, the process $(Z_t)_{t\ge 0}$ is called supercritical, critical and subcritical accordingly as $\alpha<0$, $\alpha=0$ and $\alpha>0$.
In this paper, we restrict our attention to the cases when
the CB processes hit $0$ with probability $1$,
that is, those critical or subcritical CB processes with branching mechanism $\psi$ satisfying that $\int^{+\infty}1/\psi(\lambda)\mathrm{d}\lambda<+\infty$.

We are concerned with the stationary measures of CB processes. Since $0$ is an absorbing state, the unique (up to a constant multiple) stationary measure on the state space $[0,+\infty)$ is the Dirac measure at $0$
(cf. \cite[P23--24]{Ha} ).
Therefore, we shall exclude the state $0$, and call a Radon measure $\nu$ on $(0,+\infty)$ a stationary measure for $(Z_{t})_{t\ge 0}$ if for any $t>0$ and any Borel set $A\subset (0,+\infty)$,
\[
\pp_\nu(Z_t\in A)=\nu(A),
\]
where $\pp_\nu(Z_t\in A)=
\int_{(0,+\infty)}\pp_{x}(Z_{t}\in A)\nu(\mathrm{d}x)$.
It is well-known that a CB process can be viewed as the analogue of
the Galton-Watson process
(GW process) in continuous time and continuous state space.
Before anything starts, let us first review some classical results concerning stationary measures for GW processes.
A standard reference is Athreya and Ney \cite{AN}. (see, also,
Asmussen and Hering\cite{AH},
Hoppe \cite{H}, Nakagawa \cite{NT} and Ogura and Shiotani \cite{OS} for related discussions for multitype GW processes.)
Suppose $(Y_{n})_{n\ge 0}$ is a GW process taking values in $\mathbb{Z}_{+}=\{0,1,2,\cdots\}$ with offspring distribution
$(p_k)_{k\ge 0}.$
Let $m=\sum_{j=1}^{+\infty}j p_j$ be the reproduction mean and $q=\mathrm{P}\left(Y_{n}=0\mbox{ eventually }|Y_{0}=1\right)$ be the extinction probability.
Unless $p_{1}=1$, $q<1$ iff $m>1$ (supercritical case). Hence extinction occurs almost surely in the critical
($m=1,p_1<1$)
and subcritical ($m<1$) cases.  We call $(\eta_i)_{i\ge 1}$ a stationary measure for $(Y_{n})_{n\ge 0}$ if $\eta_{i}\ge 0$ for all $i\ge 1$, and
\[
\eta_j=\sum_{i=1}^\infty \eta_i P(i,j),\qquad j\geq 1,
\]
where $(P(i,j))_{i,j\geq 0}$ denote the one step transition probabilities of $(Y_n)_{n\ge 0}$.
\cite[Theorem II.1.2]{AN} tells us that $(\eta_i)_{i\ge 1}$ is a stationary measure iff its generating function
$U(s)=\sum_{i=1}^{+\infty}\eta_is^i$ is analytic for $|s|<q$, and satisfies
Abel's equation
\begin{eqnarray*}\label{Eq Abel}
U(f(s))=U(p_0)+U(s),\qquad |s|<q,
\end{eqnarray*}
where $f$ is the generating function of the offspring distribution
$(p_k)_{k\ge 0}.$
In the supercritical case, if $q=0$, the only stationary measure is $\eta_i=0$ for all $i\ge 1$, otherwise if $q>0$, then the construction of stationary measures can be handled by reduction to the subcritical case, see, \cite[II.2]{AN}. Now we focus on the critical and subcritical cases.
It is proved in \cite{AN} that in the critical case
a (nontrivial) stationary measure
exists and is unique (up to a constant multiple), while in the subcritical case the stationary measure is not unique. In fact, in the critical case, the stationary measure
 is determined by the ratio limit of the $n$-step transition probabilities (cf. \cite[Lemma I.7.2 and Theorem II.2.1]{AN} and \cite{P}).
The continuous-time analogue of this result is due to \cite[Lemma 7]{I}.
In the subcritical case, the problem of determining all stationary measures is settled by Alsmeyer and R\"{o}sler \cite{AR}, where it is proved that
every stationary measure has
a unique integral representation in terms of the Martin entrance boundary and a finite measure on $[0,1)$.

 For a continuous-time GW process with transition functions
$\{p_{ij}(t):\ t\ge0,\ i,j\in\mathbb{Z}^+\}$,
 a stationary measure is a set of nonnegative numbers $\{\nu_j:j\ge 1\}$ satisfying that
 $$
 \nu_j=\sum_{i\ge 1}\nu_ip_{ij}(t),\quad j\ge 1,\ t\ge 0.
 $$
 In contrast to the discrete-time situation,
 for continuous-time GW process, a nontrivial stationary measure
 exists and is unique (up to a constant multiple) in both critical and subcritical cases, see, \cite[Lemma 7]{I} for the critical case and \cite[Corollary 8]{M} for the subcritical case. A similar phenomenon happens for CB processes, see Ogura\cite{Y}.
 Namely, assuming extinction occurs almost surely, the CB process has a unique nontrivial stationary measure.
 Indeed, Ogura has established the functional equation satisfied by the Laplace transform of the stationary measure (see, \cite[Lemma 1.2]{Y}), which can be viewed as the continuous counterpart of the above Abel's equation.

 In this paper, we are interested in the description of stationary measure of CB process from different points of view. We extend
 Ogura's results in the following three aspects: Firstly, we establish a representation of the stationary measure for CB processes in terms of the so called scale functions of the related L\'evy processes (Theorem \ref{them1}). Secondly, we prove that the transition probability on $(0,+\infty)$ of the CB process, when appropriately normalized, converges vaguely, and we obtain the precise limit measure (Theorem \ref{thm con of semigroup}). We shall see from this result that, in the critical case, the stationary measure can be obtained from the vague limit of an appropriately normalized transition probability of the CB process, giving an analogue of the ratio limit
 theorem (cf. \cite[Lemma I.7.2]{AN}). We remark that more regularity properties of the transition probabilities were investigated in \cite{CLP,Y1,Y2} for CB processes (with or without immigration), under additional analytical assumptions on the branching mechanisms.
 Finally, we obtain a representation of the potential measure of the CB process in terms of the scale functions, and we
 prove that the stationary measure can also be obtained from the vague limit of the potential measure in both
 critical and subcritical cases (Theorem \ref{thm con of green}).
 In the context of GW processes, result of this type is obtained in \cite{OS} for the critical case (under additional assumptions
  on the reproduction law) and in \cite{AR} for the subcritical case.
  Our proof is based on the relation between CB processes and L\'{e}vy processes through the so called Lamperti transform (see, Section \ref{sec2.1} below),
  and is easier
  than the proofs for the discrete state situation.
  Furthermore, we give equivalent conditions, depending on the branching mechanisms, for the potential measures to be finite (Proposition \ref{prop:potential}).

In this paper we also aim at linking the stationary measure to some conditional limit theorems of CB processes.
Conditional limit theorems constitute an important part of the limit theory of branching processes. There has been a lot of work on various conditional limit theorems for branching processes, see, for example, \cite{AN,E,AGP} for
discrete state situation,
and \cite{L,Li2,Pakes,RYZ} for
continuous state situation.
Suppose $(Z_t)_{t\ge 0}$ is a CB process which becomes extinct almost surely.
A usual conditioning
is made on extinction after some time $t$.  Let $\zeta$ be the extinction time. The asymptotic behavior of $Z_t$ conditioned on $\{\zeta>t\}$ is described in the so-called Yaglom's theorem. Namely, in the subcritical case, there is a probability measure $\rho$
on $(0,+\infty)$,
called the Yaglom distribution, such that for any $x>0$
and any Borel set $A\subset (0,+\infty)$,
\begin{equation}\label{eq:Yaglom}
\lim_{t\to+\infty}\pp_x(Z_t\in A|\zeta>t)=\rho(A).
\end{equation}
The Yaglom distribution belongs to the family of quasi-stationary distributions of CB processes.
A brief review of the latter
is given in the end of Section \ref{sec:II}.
By contrast,
the critical case is degenerate since all the limits on the left hand side of \eqref{eq:Yaglom} are $0$.
However, by taking different conditioning instead of conditioning on non-extinction,
one may get non-degenerate results for both critical and subcritical cases.
In Section \ref{sec:IV}, we consider two special conditioning events: $\{t\le \zeta<t+s\}$ and $\{\zeta=t\}$ ($t,s>0$). The former is regarded as conditioning on extinction in the near future $[t,t+s)$ and the latter as conditioning on extinction at time $t$.
When the extinction time $\zeta$ is finite almost surely, the event $\{t\le \zeta<t+s\}$ is of positive probability and this conditioning can be made in the usual sense. But $\{\zeta=t\}$ is of $0$ probability, and this conditioning is made
by taking limit of the conditional probability on $\{t\le \zeta<t+s\}$ as $s\to 0+$, or equivalently, by taking a Doob $h$-transform.
The study of CB process conditioned on $\{\zeta=t\}$ dates back to \cite{AD}, in which it was shown
that CB process has a spinal
decomposition, called Williams decomposition, under such a conditional probability.
Later,
similar property
for superprocesses was studied in \cite{DH,RSZ}.
For GW processes, similar conditioning is studied by Esty \cite{E}. We remark that
Esty \cite{E}
considers only critical GW processes, while we allow the CB process to be either critical or subcritical.

In this paper we prove
some limit theorems for CB processes conditioned on the aforementioned two events.
Our two principal results, Theorem \ref{thm:conditioned on near extinction} and Theorem \ref{thm for conditioned on fix time}, show that the distributions of $Z_{t-q}$ ($0<q<t$) conditioned on extinction in the near future $[t-q,t)$ and on extinction at time $t$, are convergent as $t$ goes to infinity, and we also obtain the precise limit distributions.
From these results, we shall see that the limit distributions obtained in the critical (resp. subcritical) case are of the size-biased type of the stationary measure (resp. the Yaglom distribution). As a by-product, in the critical case, we prove that the limit distribution of $Z_{t-q}$ ($0<q<t$) conditioned on $\{\zeta=t\}$ is of the size-biased type of the stationary measure, giving an analogue of \cite[Theorem I.8.2]{AN}.
Our proofs of the conditional limit theorems are based on the asymptotic estimates of the Log-Laplace functional of CB process derived from the integral equations it satisfies.
Moreover, we investigate properties of the limit distribution of $Z_{t-q}$ conditioned on extinction at time $t$.
We show that the limit is
infinitely divisible and give a representation of its L\'{e}vy-Khintchine triplet in terms of the scale functions (Proposition \ref{prop4.5}). In the subcritical case, we prove that it is weakly convergent as $q\to+\infty$ to a non-degenerate distribution under an additional $L\log L$ condition (Proposition \ref{prop V_infty}). As an application of these results, we present a new proof of a limit theorem for the CB process conditioned on non-extinction (Proposition \ref{prop1}).

We notice that by conditioning a supercritical CB process to be extinct, one recovers a subcritical CB process. To be more specific, if $\gamma$ is the largest root of $\psi(\lambda)=0$, then $\gamma>0$ in the supercritical case, and the supercritical CB process with branching mechanism $\psi$ conditioned on its extinction turns out to be a subcritical CB process with branching mechanism $\psi^*(\lambda)=\psi(\lambda+\gamma)$.
As a consequence, our conditional limit theorems obtained for
the subcritical case can be applied to supercritical CB processes conditioned to be extinct.

The remainder of this paper is organized as follows.
In Section \ref{sec:II} we recall the definition of CB processes and review some classical results concerning CB processes and L\'{e}vy processes. Then we give a representation of the stationary measure in terms of the scale functions of the related L\'{e}vy process.
In Section \ref{sec:III}, we prove the vague convergence of the normalized transition probabilities and potential measures of CB processes. Some examples are given to illustrate the results obtained in this section.
In Section \ref{sec:IV},
we study the probabilities of $Z_t$ conditioned on extinction in the near future and on extinction at a fixed time, prove some conditional limit theorems and explore some properties of the limit distributions.
Some minor statements needed along the way are proved in the Appendix.

Throughout this paper, we use ``:=" as a way of definition. For positive functions $f, g$ on $(0,+\infty)$ and constant $c\in [0,+\infty)$, we write $f(x)\sim g(x)$ as $x\to c$ if $\lim_{x\to c}f(x)/g(x)=1$. For a measure $\mu$ on $(0,+\infty)$ and
a measurable function $f$, we write $\langle f,\mu\rangle$ for the integral $\int_{(0,+\infty)}f(x)\mu(\mathrm{d}x)$.
Suppose $\nu_{n},\nu$ are measures on $(0,+\infty)$. $\nu_{n},\nu$ can be extended to measures on the larger space $[0,+\infty)$ by setting $\nu_{n}(\{0\})=\nu(\{0\})=0$. We define the vague convergence following \cite{BW}: $\nu_{n}$ is said to converge vaguely to $\nu$ if $\int_{[0,+\infty)}g(y)\nu_{n}(\mathrm{d}y)\to \int_{[0,+\infty)}g(y)\nu(\mathrm{d}y)$ for all continuous functions $g$ on $[0,+\infty)$ vanishing at infinity. If $\nu_{n},\nu$ are finite measures and $\langle f,\nu_{n}\rangle\to\langle f,\nu\rangle$ for all bounded continuous functions $f$ on $(0,+\infty)$, we say $\nu_{n}$ converges weakly to $\nu$.

\section{Preliminaries}\label{sec:II}
\subsection{CB processes and L\'{e}vy processes}\label{sec2.1}
Let $((Z_t)_{t\geq 0}, \pp_x)$ be the CB process with branching mechanism $\psi(\lambda)$ given in \eqref{def:psi}
and initial value $x>0$.  Following \cite{Kyprianou}, such a process is a time-homogeneous strong Markov process
taking values in $[0,+\infty)$ with an absorbing state $0$, such that for any $\lambda>0$,
\begin{equation}\label{eq:log-laplace}
\pp_x\left[\e^{-\lambda Z_t}\right]=\e^{-xu_t(\lambda)},\qquad t\ge 0,
\end{equation}
where $u_t(\lambda)$ is the solution to the following ordinary differential equation
\begin{equation}\label{eq:ODE}
\begin{cases}
\dfrac{\partial u_t(\lambda)}{\partial t}=-\psi(u_t(\lambda)),\\
u_0(\lambda)=\lambda.
\end{cases}
\end{equation}
We assume that $\psi(+\infty)=+\infty$. Thus by \cite[Theorem 12.3]{Kyprianou}) $(Z_t)_{t\ge 0}$ is conservative in the sense that
$\pp_x\left(Z_t<+\infty\right)=1$
for all $x>0$ and $t\ge 0$. \cite[Chapter $3$]{Li} is also a good reference for continuous state branching processes.

Let $\zeta:=\inf\{t>0:\ Z_t=0\}$  be the extinction time.
It follows by \eqref{eq:log-laplace} that
\begin{equation}\label{eq:extinction finite time}
\pp_x\left(\zeta\le t\right)=\pp_x\left(Z_t=0\right)=\e^{-x u_t(+\infty)},\quad\forall x,t>0.
\end{equation}
Let $q(x):=\pp_x\left(\zeta<+\infty\right)$ for $x> 0$.  It is proved in \cite{Grey} that
$q(x)>0$ for some (and then all) $x>0$ if and only if
\begin{equation}\label{extinct assumption}
\int^{+\infty}\frac{1}{\psi(\lambda)}\mathrm{d}\lambda<+\infty.
\end{equation}
In this case $q(x)=\e^{-x\gamma}$, where
\begin{eqnarray}
\nonumber
\gamma:=\sup\{\lambda\ge 0:\ \psi(\lambda)=0\}.
\end{eqnarray}
We know that $\psi$ is strictly convex and infinitely differentiable
on $(0,+\infty)$ with $\psi(0)=0$, $\psi(+\infty)=+\infty$ and $\psi'(0+)=\alpha$.
So we have $\gamma>0$ if $\alpha<0$ (supercritical case)
and $\gamma=0$ if $\alpha\ge 0$ (critical and subcritical cases).

Assuming \eqref{extinct assumption} holds, we can define a strictly decreasing function $\phi$ on $(\gamma,+\infty)$ by
\[
\phi(\lambda):=\int_\lambda^{+\infty}\frac{1}{\psi(u)}\mathrm{d}u, \qquad
\lambda>\gamma.
\]
It is easy to see that $\phi(\gamma)=+\infty$ and $\phi(+\infty)=0$.
Let $\varphi$ be the inverse function of $\phi$, which is defined on $(0,+\infty)$ and takes values in $(\gamma,+\infty)$.
From \eqref{eq:ODE} we have
\begin{equation}\label{integral equation}
\int_{u_t(\lambda)}^\lambda\dfrac{1}{\psi(u)}\mathrm{d}u=t,\qquad \lambda,t>0.
\end{equation}
By letting $\lambda\to+\infty$, we have
\begin{equation}\label{eq:1.6}
\int_{u_t(+\infty)}^{+\infty}\frac{1}{\psi(u)}\mathrm{d}u=t.
\end{equation}
 Recall that
 $u_t(+\infty)=-\log \pp_1(\zeta\le t)\ge -\log\pp_{1}(\zeta<+\infty)=\gamma$.
 One gets by \eqref{eq:1.6} that $u_t(+\infty)=\varphi(t)$ for all $t>0$, and consequently,
\begin{equation}\label{prob of extintion at t}
\pp_x(\zeta\leq t)=\e^{-x\varphi(t)}, \qquad  x,t>0.
\end{equation}
Particularly, if
$(Z_t)_{t\ge 0}$
is critical or subcritical, then $\gamma=0$ and \eqref{integral equation} yields
that
\begin{equation}\label{rela of u and varphi}
u_t(\lambda)=\varphi(t+\phi(\lambda)),\qquad \lambda, t>0.
\end{equation}

We note that $\psi$ is also the Laplace exponent of a spectrally positive L\'{e}vy process $(X_t)_{t\ge 0}$.
We denote by $\p_x$ the law of $(X_t)_{t\ge 0}$
started at $x\in\R$
at time $0$.
Then
$$
\p_x\left[\e^{-\lambda X_t}\right]=
\e^{-\lambda x+\psi(\lambda)t},
\qquad\lambda,t\ge 0.
$$

Define $\tau^-_0:=\inf\{t\ge 0:\ X_t<0\}$ with the convention that $\inf\emptyset=+\infty$. There is a sample-path relationship between the CB process $(Z_t)_{t\ge 0}$ and the L\'{e}vy process $(X_t)_{t\ge 0}$ stopped at $\tau^-_0$, called the Lamperti transform (cf. \cite[Theorem 12.2]{Kyprianou}
or \cite{Bin}).
Define for $t\ge 0$,
$$
 \theta_t:=\inf\left\{s>0:\ \int_0^s\frac{1}{X_u}\mathrm{d}u>t\right\}.
$$
Then $\left((X_{\theta_t\wedge \tau^-_0})_{t\ge 0},\p_x\right)$
is a CB process with branching mechanism $\psi$ and initial value
$x>0$.
We refer to \cite[Chapter 12]{Kyprianou} for the results on the long-term behavior of
CB process
based on the fluctuation theory of spectrally positive L\'{e}vy process.

\subsection{Representation of the stationary measure}
In what follows and for the remainder of this paper, we assume $(Z_t)_{t\ge 0}$ is a CB process with branching mechanism $\psi$ satisfying \eqref{extinct assumption} and
$\psi'(0+)=\alpha\ge 0$.
In this subsection we shall give a representation of the stationary measure of
$(Z_{t})_{t\ge 0}$
in terms of the so called scale function.
Recall that the scale function $W$ is the unique strictly increasing and positive continuous function on $[0,+\infty)$ such that
\begin{equation}\label{Laplace of scale function}
\int_0^{+\infty} \e^{-\lambda x} W(x)\mathrm{d}x=\frac{1}{\psi(\lambda)},\qquad
\lambda>0.
\end{equation}
We define $W(x)=0$ for $x<0$.
We refer to \cite[Chapter VII]{Bertoin} and \cite{KKR} for the general theory of scale functions.

We write $\int_{0+}^{+\infty}$ for $\int_{(0,+\infty)}$ to emphasize the integral is on $(0,+\infty)$.
For a measure $\nu$ on $(0,+\infty)$, we set $\widehat{\nu}(\lambda):=\int_{0+}^{+\infty}\e^{-\lambda x}\nu(\mathrm{d}x)$ for $\lambda\ge 0$
whenever the right hand side is well-defined.

\begin{thm}\label{them1}
Set
\begin{equation}
\nonumber
\mu(\mathrm{d}x):=\frac{W(x)}{x}\mathrm{d}x
\quad\mbox{ for }x>0.
\end{equation}
Then $\mu(\mathrm{d}x)$ is the unique (up to a constant multiple) stationary measure for $(Z_t)_{t\ge0}$.
\end{thm}

\begin{proof}
Due to \cite[Lemma 1.2 and Proposition 1.3]{Y}, for a CB process which satisfies \eqref{extinct assumption}, there exists a unique (up to a constant multiple) stationary measure
$\varrho$
such that
\begin{equation}
\nonumber
\widehat{\varrho}(\lambda)
=\phi(\lambda),
\quad \lambda>\gamma.
\end{equation}
Recall that the CB process we consider in this theorem is
critical or subcritical,  and then $\gamma=0$.
We have for any $\lambda>0$,
\begin{eqnarray*}
\int_{0+}^{+\infty} \e^{-\lambda x}\mu(\mathrm{d}x)
&=&\int_{0+}^{+\infty}\mu(\mathrm{d}x)\int_{\lambda}^{+\infty}x\e^{-ux}\mathrm{d}u\\
&=&\int_0^{+\infty} W(x) \mathrm{d}x \int_{\lambda}^{+\infty} \e^{-ux} \mathrm{d}u \\
&=&\int_{\lambda}^{+\infty}\mathrm{d}u\int_{0}^{+\infty}W(x)\e^{-ux}\mathrm{d}x\\
&=&\int_{\lambda}^{+\infty} \frac{1}{\psi(u)}\mathrm{d}u
=\phi(\lambda).
\end{eqnarray*}
Hence $\mu$ is the unique (up to a constant multiple) stationary measure.
\end{proof}

\begin{remark}
Suppose $(Z_t)_{t\ge 0}$ is a supercritical CB process satisfying \eqref{extinct assumption} and $\gamma=\sup\{\lambda\ge 0:\ \psi(\lambda)=0\}>0$.
Repeating the calculation in the proof of Theorem \ref{them1},
we can show that, for $\mu(\mathrm{d}x)=\frac{W(x)}{x}\mathrm{d}x$,
\begin{equation}
\nonumber
\int_{0+}^{+\infty}\e^{-\lambda x}\mu(\mathrm{d}x)=\phi(\lambda),
\quad \lambda>\gamma.
\end{equation}
Hence the result of Theorem \ref{them1}
also holds for this supercritical CB process.
\end{remark}

We notice that $\widehat{\mu}(0)=\phi(0)=+\infty$. So $\mu$ is an infinite measure on $(0,+\infty)$. Theorem \ref{them1} implies that the CB process has no stationary distributions on $(0,+\infty)$. Instead, one may consider a subinvariant distribution, called the quasi-stationary distribution (QSD). For a CB process, a QSD is a probability measure $\nu$ on $(0,+\infty)$ satisfying that
\begin{equation}\label{eq:def-quasistationary}
\pp_{\nu}\left(Z_t\in A\,|\,\zeta>t\right)=\nu(A)
\end{equation}
for any Borel set $A\subset (0,+\infty)$ and $t>0$.
One can easily show by the Markov property that
$$
\pp_{\nu}(\zeta>t+s)=\pp_{\nu}(\zeta>t)
\pp_{\nu}(\zeta>s),\qquad t,s\ge 0.
$$
Hence the extinction time $\zeta$ under $\pp_{\nu}$
is exponentially distributed with some parameter $\beta>0$. So \eqref{eq:def-quasistationary} is equivalent to
$$
\pp_{\nu}(Z_{t}\in A)=\e^{-\beta t}\nu(A)
$$
for any Borel set $A\subset (0,+\infty)$ and $t>0$.
A discrete state analogue
is the so called $\lambda$-invariant measure, for which we refer to \cite{M}.
Lambert \cite{L}
has given a complete characterization of QSD's for CB processes.
It is proved in \cite{L}
that a subcritical CB process has QSD's while a critical CB process has no QSD. In fact, for a subcritical CB process with $\psi'(0+)=\alpha>0$,  all QSD's form a stochastically decreasing family $\{\nu_{\beta}\}$ of probabilities indexed by $\beta\in (0,\alpha]$ satisfying that
\begin{equation}\label{lap of quasi}
\widehat{\nu}_{\beta}(\lambda)=1-\e^{-\beta\phi(\lambda)},
\qquad \lambda>0.
\end{equation}
The probability $\nu_{\alpha}$ is the so-called Yaglom distribution in the sense that
\begin{equation}\label{eq:yaglom limit}
\lim_{t\to+\infty}\pp_x\left(Z_t\in A\,|\,\zeta>t\right)=\nu_{\alpha}(A)
\end{equation}
for every $x>0$ and Borel set $A\subset (0,+\infty)$.
The conditional limit of \eqref{eq:yaglom limit} is due to Li \cite[Theorem 4.3]{Li2},
where more general conditioning of the type $\{\zeta>t+r\}$ with $r\geq 0$ is considered.
From the theory of Laplace transform,
the QSD $\nu_{\beta}$ can be expressed by the stationary measure
$\mu$ as
\[
\nu_\beta(\mathrm{d}x)=-\sum_{n=1}^{+\infty} \frac{(-\beta)^n}{n!}\mu^{*n}(\mathrm{d}x),
\]
where $\mu^{*n}$ denotes the $n$ fold convolution of $\mu$.
On the other hand, since $\widehat\nu_{\beta}(\lambda)/\beta\to \phi(\lambda)$ as $\beta\to 0+$ for all $\lambda>0$, we get that
$\frac{1}{\beta}\nu_{\beta}$ converges vaguely to $\mu$ as $\beta\to 0+$.

Though there is no QSD in the critical case,
convergence results are established for the rescaled process $Q_{t}Z_{t}$ conditioned on $\{\zeta>t\}$, where $Q_t\to 0$ as $t\to+\infty$.
It is proved by \cite[Theorem 5.2]{Li2}
that if the critical CB process has finite variance, that is, $\psi''(0+)<+\infty$, then $Z_t/t$ conditioned on
$\zeta>t$ converges in distribution to an exponential distribution random variable with parameter $2/\psi''(0+)$.
We refer to \cite{RYZ} for the case allowing infinite variance.

\section{Convergence of transition probabilities and potential measures}\label{sec:III}
Let $(P_t(x,\mathrm{d}y); t\geq 0, x,y\geq 0)$ be the transition probability of the CB process $(Z_t)_{t\ge 0}$.  Firstly, we shall show that the transition
probability $P_t(x,\mathrm{d}y)$ on $(0,+\infty)$, when appropriately normalized, converges vaguely to a precise measure.  For notational simplification, we
still
use $P_t(x,\mathrm{d}y)$ to denote the restriction of $P_t(x,\mathrm{d}y)$ on $(0,+\infty)$.

\begin{lem}\label{lem3.1}
If $\alpha=0$, then
\begin{equation*}
\lim_{t\to+\infty}\frac{\varphi(t)-\varphi(t+\phi(\lambda))}{\psi(\varphi(t))}=
\phi(\lambda),\qquad\lambda>0.
\end{equation*}
\end{lem}

\begin{proof}
It follows by the monotone convergence theorem that
$$
\psi'(\lambda)=\sigma^2\lambda+\int_0^{+\infty} (1-\e^{-\lambda r})r\pi(\mathrm{d}r)\to 0\quad\mbox{ as }\lambda\to 0+.
$$
We note that
$(\psi(\varphi(t)))'=-\psi'(\varphi(t))\psi(\varphi(t))$ for $t>0$ and that
$t\mapsto\varphi(t)$ is strictly decreasing on $(0,+\infty)$ with $\varphi(+\infty)=0$.
Thus for any $s>0$,
\begin{equation}\nonumber
\ln\frac{\psi(\varphi(t+s))}{\psi(\varphi(t))}=-\int_t^{t+s}\psi'(\varphi(u))\mathrm{d}u\to 0\quad\mbox{ as }t\to+\infty.
\end{equation}
It follows that
\begin{equation}\label{ratio limit}
\lim_{t\to+\infty}\frac{\psi(\varphi(t+s))}{\psi(\varphi(t))}=1.
\end{equation}
By the mean value theory, for every $t>0$ and $\lambda>0$ there exists $\Delta_t(\phi(\lambda))\in[0,\phi(\lambda)]$ such that
\begin{equation}\label{lem3.1.2}
\frac{\varphi(t)-\varphi(t+\phi(\lambda))}{\psi(\varphi(t))}=\frac{\psi(\varphi(t+\Delta_t(\phi(\lambda))))}{\psi(\varphi(t))}\phi(\lambda).
\end{equation}
Since $t\mapsto\psi(\varphi(t))$ is strictly decreasing on $(0,+\infty)$, we have
$$
\frac{\psi(\varphi(t+\phi(\lambda)))}{\psi(\varphi(t))}\le \dfrac{\psi(\varphi(t+\Delta_t(\phi(\lambda))))}{\psi(\varphi(t))}\le 1.
$$
By \eqref{ratio limit},
$$
\lim_{t\to+\infty}\dfrac{\psi(\varphi(t+\Delta_t(\phi(\lambda))))}{\psi(\varphi(t))}=1.
$$
Combining this with \eqref{lem3.1.2} we get
\[
\lim_{t\to+\infty}\frac{\varphi(t)-\varphi(t+\phi(\lambda))}{\psi(\varphi(t))}
=\lim_{t\to+\infty}\frac{\psi(\varphi(t+\Delta_t(\phi(\lambda))))}{\psi(\varphi(t))}\phi(\lambda)=\phi(\lambda).
\]
\end{proof}

\begin{thm}\label{thm con of semigroup}
If $\alpha>0$, then for every $x>0$, $\frac{1}{x\psi(\varphi(t))}P_t(x,\mathrm{d}y)$
converges weakly to $\frac{1}{\alpha}\nu_{\alpha}(\mathrm{d}y)$ as $t\to+\infty$.  Otherwise if $\alpha=0$, then for every $x>0$, $\frac{1}{x\psi(\varphi(t))}P_t(x,\mathrm{d}y)$
converges vaguely to $\mu(\mathrm{d}y)$ as $t\to+\infty$.
\end{thm}
\begin{proof}
By Lemma \ref{A1}, it suffices to show that for any $x>0$,
\begin{equation}\label{them3.2.1}
\lim_{t\to+\infty}\frac{1}{\psi(\varphi(t))}\int_{0+}^{+\infty}\e^{-\lambda y}P_t(x,\mathrm{d}y)=
\begin{cases}
\frac{x}{\alpha}\widehat{\nu}_{\alpha}(\lambda), \,\,\quad\forall \lambda\ge 0,&\quad\mbox{if}\quad \alpha>0;\\
x\widehat{\mu}(\lambda),\qquad\forall \lambda>0,& \quad\mbox{if}\quad \alpha=0.
\end{cases}
\end{equation}
For any
$\lambda\ge 0$,
we have
$$
\int_{0+}^{+\infty}\e^{-\lambda y}P_t(x,\mathrm{d}y)=\pp_x\left[\e^{-\lambda Z_t},Z_t>0\right]=\e^{-xu_t(\lambda)}-\e^{-x\varphi(t)}.
$$
If $\alpha>0$, then $\psi(\varphi(t))\sim \alpha\varphi(t)$ and
$\pp_{x}\left(\zeta>t\right)=1-\e^{-x\varphi(t)}\sim x\varphi(t)$ as $t\to+\infty$.
Consequently, we have, for $\lambda\ge 0$,
$$
\frac{1}{\psi(\varphi(t))}\int_{0+}^{+\infty}\e^{-\lambda y}P_{t}(x,\mathrm{d}y)\sim \frac{x}{\alpha}\pp_{x}\left[\e^{-\lambda Z_{t}}|\zeta>t\right]\mbox{ as }t\to+\infty.
$$
Thus, if  $\alpha>0$ the result \eqref{them3.2.1}
is a reformulation of the limit \eqref{eq:yaglom limit}.
Now suppose $\alpha=0$. By Lemma \ref{lem3.1}, we have for any $\lambda>0$,
\begin{eqnarray*}
&&\lim_{t\to+\infty}\frac{1}{\psi(\varphi(t))}\int_{0+}^{+\infty} \e^{-\lambda y}P_t(x,\mathrm{d}y)
=\lim_{t\to+\infty}\frac{\e^{-x\varphi(t+\phi(\lambda))}-\e^{-x\varphi(t)}}{\psi(\varphi(t))}\nonumber\\
&=&\lim_{t\to+\infty}\frac{x(\varphi(t)-\varphi(t+\phi(\lambda)))}{\psi(\varphi(t))}
=x\phi(\lambda)=x\widehat{\mu}(\lambda).
\end{eqnarray*}
Hence we prove \eqref{them3.2.1}.
\end{proof}

Theorem \ref{thm con of semigroup} implies that the transition probability $P_t(x,\mathrm{d}y)$
constrained on $(0,\infty)$
is vaguely convergent
with rate $x\psi(\varphi(t))$ as $t\to+\infty$. In the following we shall
give concrete examples to illustrate the result of Theorem \ref{thm con of semigroup}.

\begin{exa}
Suppose $(Z_t)_{t\ge 0}$ is a subcritical CB process with $\psi'(0+)=\alpha>0$.
Let $\Theta$ be
a positive random variable
whose distribution is equal to
the Yaglom distribution $\nu_{\alpha}$.
By \cite[Lemma 2.1]{L}, $\mathrm{E}[\Theta]<+\infty$ if and only if
\begin{equation}\label{con2}
\int^{+\infty}r\ln r\pi(\mathrm{d}r)<+\infty,
\end{equation}
and in this case
$\varphi(t)\sim \frac{1}{\mathrm{E}[\Theta]}\e^{-\alpha t}$
as $t\to+\infty$.  Thus,
$$
\psi(\varphi(t))\sim \psi'(0+)\varphi(t)\sim  \frac{\alpha}{\mathrm{E}[\Theta]}\e^{-\alpha t}\quad\mbox{ as }t\to+\infty.
$$
Theorem \ref{thm con of semigroup} yields that for every $x>0$,
restricted on $(0,+\infty)$,
$$
\e^{\alpha t}P_t(x,\mathrm{d}y)
\mbox{ converges weakly to }\frac{x}{\mathrm{E}[\Theta]}\nu_{\alpha}(\mathrm{d}y)\mbox{ as }t\to+\infty.
$$
Otherwise if \eqref{con2} fails, then $\varphi(t)=o(\e^{-\alpha t})$ and thus $\psi(\varphi(t))=o(\e^{-\alpha t})$. Hence
$\e^{\alpha t}P_t(x,\mathrm{d}y)$
converges weakly to the null measure.
\end{exa}

\begin{exa}
Suppose $(Z_t)_{t\ge 0}$ is a critical CB process with branching mechanism $\psi$ given by
\begin{equation*}
\psi(\lambda)=\lambda^{1+p}L\left(1/\lambda\right),\qquad \lambda>0,
\end{equation*}
where $0<p\le 1$
and $L$ is a slowly varying function at $+\infty$.  For a slowly varying function $l$, it is known
(cf. \cite[Theorem 1.5.13]{BGT}) that there exists
a unique
(up to asymptotic equivalence) slowly varying function $l^{\#}$ such that $l(x)l^{\#}(xl(x))\to 1$ and $l^{\#}(x)l(xl^{\#}(x))\to 1$ as $x\to+\infty$. $l^{\#}$ is called the de Brujin conjugate of $l$.

For $z>0$, let $g(z):=\phi(1/z)=\int_{1/z}^{+\infty}\frac{1}{\psi(\lambda)}\mathrm{d}\lambda=\int^{z}_{0}\frac{u^{p-1}}{L(u)}\mathrm{d}u$. Since $p-1>-1$, by Karamata's theorem (cf.\cite[Theorem 1.5.11]{BGT}),
$$
g(z)\sim \frac{z^{p}}{p L(z)}\quad\mbox{ as }z\to +\infty.
$$
Note that $g$ is a strictly increasing function on $(0,+\infty)$. Let $g^{-1}$ be its inverse. It follows by \cite[Proposition 1.5.15]{BGT} that
$$
g^{-1}(z)\sim p z^{1/p}L^{\diamondsuit}(z^{1/p})^{1/p}\mbox{ as }z\to+\infty,
$$
where $L^{\diamondsuit}$
is the de Brujin conjugate of $1/L$. Recall that $\varphi(t)=\phi^{-1}(t)=1/g^{-1}(t)$. We get
$$
\varphi(t)\sim \frac{1}{p} t^{-1/p}L^{\diamondsuit}(t^{1/p})^{-1/p}\mbox{ as }t\to+\infty.
$$
We note that
$$
\varphi(t)=-\int_t^{+\infty}\varphi'(s)\mathrm{d}s=\int_t^{+\infty}\psi(\varphi(s))\mathrm{d}s.
$$
We also note that $\psi(\varphi(s))$ is a strictly decreasing function on $(0,+\infty)$. Hence by the monotone density theorem (cf. \cite[Theorem 1.7.2]{BGT}),
$$
\psi(\varphi(t))\sim \frac{1}{p^{2}}t^{-\left(\frac{1}{p}+1\right)}L^{\diamondsuit}(t^{1/p})^{-1/p}\mbox{ as }t\to+\infty.
$$
Therefore Theorem \ref{thm con of semigroup} yields that for every $x>0$,
$$
\frac{p^{2}}{x}t^{\frac{1}{p}+1}L^{\diamondsuit}(t^{1/p})^{1/p}P_{t}(x,\mathrm{d}y)\mbox{ converges vaguely to }\mu(\mathrm{d}y)\mbox{ as }t\to+\infty.
$$
\end{exa}

\bigskip

We put for every $x>0$ and Borel set $A\subset (0,+\infty)$,
$$G(x,A):=\int_{0}^{+\infty}\pp_{x}(Z_{t}\in A)\mathrm{d}t\in [0,+\infty],$$
and call the corresponding measure $G(x,\mathrm{d}y)$ on $(0,+\infty)$ the potential measure of $(Z_{t})_{t\ge 0}$.
Equation \eqref{them3.2.1} yields that,
if $\alpha\geq 0$ (subcritical or critical case), for every $x>0$ and $\lambda>0$,
\begin{equation}\label{eq3.12}
\int_{0+}^{+\infty}\e^{-\lambda y}P_t(x,\mathrm{d}y)\sim c_{\lambda}x\psi(\varphi(t))\mbox{ as }t\to +\infty,
\end{equation}
for some positive constant $c_{\lambda}$ depending on $\lambda$.
We note that $\varphi'(t)=-\psi(\varphi(t))$. Thus
$$
\int_1^{+\infty} \psi(\varphi(t))\mathrm{d}t=\varphi(1)-\varphi(\infty)=\varphi(1)<+\infty.
$$
Hence we deduce by \eqref{eq3.12} that
$\int_{0+}^{+\infty}\e^{-\lambda y}G(x,\mathrm{d}y)
=\int_{0}^{+\infty}\int_{0+}^{+\infty}\e^{-\lambda y}P_t(x,\mathrm{d}y)\mathrm{d}t<+\infty$
for every $x>0$.  This implies that
$G(x,B)<+\infty$
for every compact subset $B\subset (0,+\infty)$.
Thus the potential measure for the CB process $(Z_t)_{t\ge 0}$ is a locally finite measure on $(0,+\infty)$.

\begin{thm}\label{thm con of green}
The potential measure $G(x,\mathrm{d}y)$
of $(Z_t)_{t\ge 0}$
has a density with respect to the Lebesgue measure given by
\begin{equation}
\nonumber
g(x,y)=\frac{W(y)-W(y-x)}{y}
\end{equation}
for $x,y>0$.
Moreover, $G(x,\mathrm{d}y)$ converges vaguely to the stationary measure $\mu(\mathrm{d}y)$ as $x\to+\infty$.
\end{thm}
\begin{proof}
Suppose $(X_t)_{t\ge 0}$ is the spectrally positive L\'{e}vy process associated with the CB process $(Z_t)_{t\ge 0}$ through the Lamperti transform. Then we have for $x>0$ and $\lambda>0$,
\begin{eqnarray}\label{them3.5.1}
\int_{0+}^{+\infty}\e^{-\lambda y}G(x,\mathrm{d}y)
&=&\pp_x\left[\int_{0}^{\zeta}\e^{-\lambda Z_t}\mathrm{d}t\right]\nonumber\\
&=&\p_x\left[\int_0^{\tau^-_0}\e^{-\lambda X_s}\frac{1}{X_s}\mathrm{d}s\right].
\end{eqnarray}
The final equality follows from a change of variables. Let $U(x,\mathrm{d}y)$ be the potential measure of $X$ killed on exiting $[0,+\infty)$ when issued from $x>0$, that is
$$
U(x,\mathrm{d}y)=\int_0^{+\infty}\p_x\left(X_t\in \mathrm{d}y,\ t<\tau^-_0\right)\mathrm{d}t\quad\mbox{ for }\, y>0.
$$
It follows by \eqref{them3.5.1} that
\begin{equation}\label{them3.5.2}
G(x,\mathrm{d}y)=\frac{1}{y} U(x,\mathrm{d}y)\quad\mbox{ for }x,y>0.
\end{equation}
It is proved in \cite[Theorem 2.7]{KKR} that $U(x,\mathrm{d}y)$ has a potential density with respect to the Lebesgue measure given by
\begin{equation}\label{eq: u}
u(x,y)=\e^{-\gamma x}W(y)-W(y-x),\qquad x,y>0.
\end{equation}
Here $\gamma=0$ since $\psi'(0+)\ge 0$.
Putting this back to \eqref{them3.5.2}, we prove the first assertion.

We note that for $\lambda>0$,
\begin{eqnarray*}
\int_{0+}^{+\infty}\e^{-\lambda y}G(x,\mathrm{d}y)
&=&\int_{0}^{+\infty}\e^{-\lambda y}\,\frac{W(y)-W(y-x)}{y}\,\mathrm{d}y\\
&=&\int_{0+}^{+\infty}\e^{-\lambda y}\mu(\mathrm{d}y)-\int_x^{+\infty}\e^{-\lambda y}\frac{W(y-x)}{y}\,\mathrm{d}y.
\end{eqnarray*}
By change of variables, the second integral in the right hand side equals
$\e^{-\lambda x}\int_{0}^{+\infty}\e^{-\lambda z}\frac{W(z)}{x+z}\mathrm{d}z$, which converges to $0$ as $x\to+\infty$.
Hence we get
$$
\lim_{x\to+\infty}\int_{0+}^{+\infty}\e^{-\lambda y}G(x,\mathrm{d}y)=\widehat{\mu}(\lambda),
$$
for all $\lambda>0$. Hence we prove the second assertion.
\end{proof}

\begin{remark}
We remark that \eqref{eq: u} holds indeed for $\gamma\ge 0$.
Thus for a supercritical CB process, applying similar argument with minor modification,
we can show that the potential density function exists and is given by
$$
g(x,y)=\e^{-\gamma x}\frac{W(y)}{y}-\frac{W(y-x)}{y},
$$
for $x,y>0$.
\end{remark}

A natural question is under what condition,
$G(x,\mathrm{d}y)$ is a finite measure on $(0,+\infty)$.  We give the following equivalent statements.

\begin{prop}\label{prop:potential}
The following statements are equivalent:
\begin{itemize}
\item[(i)]  $G(x,\mathrm{d}y)$ is a finite measure on $(0,+\infty)$ for some (and then all) $x>0$.
\item[(ii)]
 $ \pp_x[\zeta]<+\infty$ for some (and then all) $x>0$.
\item[(iii)] The branching mechanism $\psi$ satisfies that
\begin{equation}\label{finiteG}
\int_{0+}\frac{u}{\psi(u)}\mathrm{d}u<+\infty.
\end{equation}
\end{itemize}
\end{prop}
\begin{proof}
(i)$\Longleftrightarrow$(ii): By Fubini's theorem, we have for every $x>0$,
\begin{equation}\label{integral of Green function}
\int_{0+}^{+\infty} G(x,\mathrm{d}y)=\int_0^{+\infty} \mathrm{d}t\int_{0+}^{+\infty} P_t(x,\mathrm{d}y)=\int_0^{+\infty} \pp_x(\zeta>t) \mathrm{d}t=\pp_x[\zeta].
\end{equation}
Hence (i) and (ii) are equivalent.

(i)$\Longleftrightarrow$(iii):
 We have for every $x>0$,
\[
\int_{0+}^{+\infty} G(x,\mathrm{d}y)=\int_0^{+\infty} \pp_x(\zeta>t) \mathrm{d}t=\int_0^{+\infty}\left(1-\e^{-x\varphi(t)}\right)\mathrm{d}t.
\]
Since $\varphi(t)\to 0$ as $t\to+\infty$, we have $1-\e^{-x\varphi(t)}\sim x\varphi(t)$ as $t\to+\infty$.
Hence the final integral is finite if and only if $\int^{+\infty} \varphi(t)\mathrm{d}t<+\infty$.  Substituting $t$ by $\phi(s)$ in the integral
$\int^{+\infty} \varphi(t)\mathrm{d}t$,
we can deduce that $\int^{+\infty}\varphi(t)\mathrm{d}t<+\infty$ if and only if
\[
-\int_{0+}s\,\mathrm{d}\phi(s)=\int_{0+}\dfrac{s}{\psi(s)}\mathrm{d}s<+\infty.
\]
\end{proof}
We will also classify the finiteness of $G(x,\mathrm{d}y)$ through the L\'{e}vy measure $\pi$.
\begin{cor}
If $\alpha>0$, then $G(x,\mathrm{d}y)$ is a finite measure on $(0,+\infty)$ for every $x>0$.
If $\alpha=0$, then $G(x,\mathrm{d}y)$ is finite on $(0,+\infty)$ for some (then all) $x>0$ if and only if
\begin{equation}\label{finiteG2}
\int^{+\infty}\frac{1}{s\int_0^s\bar{\bar{\pi}}(r)\mathrm{d}r}\mathrm{d}s<+\infty,
\end{equation}
where for $r\ge 0$, $\bar{\pi}(r):=\int_r^{+\infty}\pi(\mathrm{d}y)$ and $\bar{\bar{\pi}}(r):=\int_r^{+\infty}\bar{\pi}(y)\mathrm{d}y$, or equivalently,
\begin{equation}\label{finiteG3}
\int^{+\infty}\frac{1}{s\int_0^sr^2\pi(\mathrm{d}r)+s^2\int_s^{+\infty} r\pi(\mathrm{d}r)}\mathrm{d}s<+\infty.
\end{equation}
\end{cor}
\begin{proof}
 If $\alpha>0$, then $u/\psi(u)\sim 1/\alpha$ as $u\to 0$, and \eqref{finiteG} holds immediately. Now we suppose $\alpha=0$. In this case
\begin{eqnarray*}
\frac{\psi(\lambda)}{\lambda}
&=&\frac{1}{2}\sigma^2\lambda+\frac{1}{\lambda}\int_0^{+\infty} (\e^{-\lambda r}-1+\lambda r)\pi(\mathrm{d}r)\\
&=&\frac{1}{2}\sigma^2\lambda+\int_{0}^{+\infty}\left(1-\e^{-\lambda r}\right)\bar{\pi}(r)\mathrm{d}r
\end{eqnarray*}
for $\lambda>0$. Obviously $\psi(\lambda)/\lambda$ is the Laplace exponent of a L\'{e}vy subordinator. Thus by \cite[Proposition III.1]{Bertoin},
$$
\frac{\psi(\lambda)}{\lambda}\asymp \lambda\left(\frac{1}{2}\sigma^{2}+\int_{0}^{1/\lambda}\bar{\bar{\pi}}(r)\mathrm{d}r\right).
$$
Consequently we have
$$
\int_{0+}\frac{u}{\psi(u)}\mathrm{d}u\asymp \int_{0+}\frac{1}{u}\cdot\frac{1}{\frac{1}{2}\sigma^{2}+\int_{0}^{1/u}\bar{\bar{\pi}}(r)\mathrm{d}r}\mathrm{d}u.
$$
By change of variables, the integral in the right hand side equals $\int^{+\infty}\frac{1}{s\left(\frac{1}{2}\sigma^{2}+\int_{0}^{s}\bar{\bar{\pi}}(r)\mathrm{d}r\right)}\mathrm{d}s$.
If $\int_0^{+\infty}\bar{\bar{\pi}}(r)\mathrm{d}r<+\infty$, then the latter integral equals $+\infty$ and \eqref{finiteG} fails. Otherwise if $\int_0^{+\infty}\bar{\bar{\pi}}(r)\mathrm{d}r=+\infty$, then
$\frac{1}{\frac{1}{2}\sigma^{2}+\int_0^s\bar{\bar{\pi}}(r)\mathrm{d}r}\sim \frac{1}{\int_0^s\bar{\bar{\pi}}(r)\mathrm{d}r}$ as $s\to +\infty$,
and \eqref{finiteG} holds if and only if \eqref{finiteG2} holds. Next, we prove the equivalence of \eqref{finiteG2} and \eqref{finiteG3}.
For any $s>0$, by exchanging the order of integration, we obtain that
\begin{eqnarray*}
&&\int_0^s\bar{\bar{\pi}}(r)\mathrm{d}r=\int_0^{+\infty} \pi(\mathrm{d}r)\int_0^r(u\wedge s)\mathrm{d}u\\
&=&\frac{1}{2}\int_0^s r^2\pi(\mathrm{d}r)+s\int_s^{+\infty} r\pi(\mathrm{d}r)-\frac{s^2}{2}\bar{\pi}(s).
\end{eqnarray*}
Note that $0\leq \bar{\pi}(s)\leq \dfrac{\int_s^{+\infty} r\pi(\mathrm{d}r)}{s}$.  These deduce the following inequalities
\[
\frac{1}{2}\int_0^s r^2\pi(\mathrm{d}r)+\frac{s}{2}\int_s^{+\infty} r\pi(\mathrm{d}r)\leq \int_0^s\bar{\bar{\pi}}(r)\mathrm{d}r
\leq \dfrac{1}{2}\int_0^s r^2\pi(\mathrm{d}r)+s\int_s^{+\infty} r\pi(\mathrm{d}r).
\]
Or it can be expressed as
\[
\int_0^s\bar{\bar{\pi}}(r)\mathrm{d}r\asymp \int_0^s r^2\pi(\mathrm{d}r)+s\int_s^{+\infty} r\pi(\mathrm{d}r).
\]
And the equivalence of \eqref{finiteG2} and \eqref{finiteG3} is obtained.
\end{proof}
From this result, we can see that if the critical CB process has finite variance, i.e. $\int_1^{+\infty} r^2\pi(\mathrm{d}r)<+\infty$, then $\pp_x[\zeta]=+\infty$ for every $x>0$, though $\pp_x(\zeta<+\infty)=1$.
However, if the right tail of the L\'evy measure $\pi$ of the critical CB process is
heavy enough,
for example, $\pi(\mathrm{d}r)=r^{-(2+p)}\mathrm{d}r$
for some $p\in (0,1)$, then the expectation of $\zeta$ is finite.

\section{CB process conditioned on extinction}\label{sec:IV}
\subsection{Existence of conditional limits}\label{sec4.1}
\begin{lem}\label{lem for U}
For any $s>0$, set
\begin{equation}\label{def of u}
\mu_s(\mathrm{d}x):=\e^{-\varphi(s)x}\frac{W(x)}{s x}\mathrm{d}x
\end{equation}
for $x>0$. Then $\mu_s$ is a probability measure on $(0,+\infty)$ with
$$
\widehat{\mu}_s(\lambda)=\frac{\phi(\lambda+\varphi({s}))}{{s}},\qquad \lambda>0.
$$
Moreover, $\mu_s$ is the size-biased stationary measure given by
$$
\mu_s(\mathrm{d}x)=\frac{\e^{-\varphi(s)x}\mu(\mathrm{d}x)}{\int_{0}^{+\infty}\e^{-\varphi(s)r}\mu(\mathrm{d}r)}.
$$
\end{lem}
\begin{proof}
By \eqref{Laplace of scale function} and Fubini's theorem, we have for $\lambda\ge 0$,
\begin{eqnarray*}
\phi(\lambda+\varphi(s))&=&\int_{\lambda+\varphi(s)}^{+\infty}\frac{1}{\psi(u)}\mathrm{d}u
=\int_{\lambda+\varphi(s)}^{+\infty} \mathrm{d}u\int_0^{+\infty} \e^{-ux}W(x)\mathrm{d}x\\
&=&\int_0^{+\infty} W(x) \mathrm{d}x \int_{\lambda+\varphi(s)}^{+\infty} \e^{-ux} \mathrm{d}u \\
&=&\int_0^{+\infty} \e^{-(\lambda+\varphi(s))x}\frac{W(x)}{x} \mathrm{d}x \\
&=&s\int_0^{+\infty}\e^{-\lambda x}\mu_s(\mathrm{d}x).
\end{eqnarray*}
In particularly if $\lambda=0$,
$\int_0^{+\infty} \mu_s(\mathrm{d}x)=\phi(\varphi(s))/s=1$.  It follows that
$\mu_s(\mathrm{d}x)$ is a probability measure on $(0,+\infty)$.
The second assertion follows immediately by
observing that
$\int_0^{+\infty}\e^{-\varphi(s)x}\mu(\mathrm{d}x)=s$.
\end{proof}

Recall that $\Theta$ is
a random variable
distributed as Yaglom distribution $\nu_\alpha$. Then its Laplace function is given by
\begin{equation}\label{Lap of Theta}
\mathrm{E}\left[\e^{-\lambda \Theta}\right]=1-\e^{-\alpha\phi(\lambda)}, \qquad \lambda>0.
\end{equation}
The following result establishes the limit distribution of CB process conditioned on extinction in the near future.

\begin{thm}\label{thm:conditioned on near extinction}
For any $s>0$, there is a positive random variable $W_s$ such that for any $\lambda,x>0$,
\[
\lim_{t\to+\infty}\pp_x\left[\left.\e^{-\lambda Z_t}\right|t\le \zeta<t+s\right]=\mathrm{E}\left(\e^{-\lambda W_s}\right)
=\begin{cases}
\dfrac{1-\e^{-\alpha\phi(\lambda+\varphi(s))}}{1-\e^{-\alpha s}},& \, \alpha>0;\\
\\
\dfrac{\phi(\lambda+\varphi(s))}{s},&\, \alpha=0.
\end{cases}
\]
In particular, if $\alpha=0$, then $W_s$ has the distribution $\p (W_s\in \mathrm{d}r)=\mu_s(\mathrm{d}r)$, where
$\mu_s$ is the size-biased stationary measure
defined in \eqref{def of u}.
Otherwise if $\alpha>0$,
then $W_s$ has the size-biased Yaglom distribution
\begin{equation}\label{them4.2.1}
\p(W_s\in \mathrm{d}r)=\frac{\e^{-\varphi(s)r}\p(\Theta\in \mathrm{d}r)}{\mathrm{E}\left[\e^{-\varphi(s)\Theta}\right]}.
\end{equation}
\end{thm}
\begin{proof}
It follows from the Markov property of $(Z_t)_{t\ge 0}$ that
\begin{eqnarray*}
\pp_x\left[\left.\e^{-\lambda Z_t}\right|t\le \zeta<t+s\right]=\frac{\pp_x\left[\e^{-\lambda Z_t}{\rm I}_{\{\zeta\ge t\}}
\pp_{Z_t}(\zeta<s)\right]}{\pp_x(\zeta<s+t)-\pp_x(\zeta<t)}.
\end{eqnarray*}
Taking use of \eqref{prob of extintion at t} and \eqref{rela of u and varphi},
we obtain that
\begin{equation}\label{conditioned expression}
\pp_x\left[\left.\e^{-\lambda Z_t}\right|t\le \zeta<t+s\right]
=\frac{\pp_x\left[\e^{-(\lambda+\varphi(s))Z_t}{\rm I}_{\{\zeta\ge t\}}\right]}{\e^{-x\varphi(t+s)}-\e^{-x\varphi(t)}}
=\frac{\e^{-x\varphi(t+\phi(\lambda+\varphi(s)))}-\e^{-x\varphi(t)}}{\e^{-x\varphi(t+s)}-\e^{-x\varphi(t)}}.
\end{equation}
When $\alpha=0$, since $\lim_{t\to+\infty}\varphi(t)=0$ and
$\varphi'(\lambda)=-\psi(\varphi(\lambda))$,
by the mean value theorem for integral,
\begin{eqnarray*}
&&\lim_{t\to+\infty}\pp_x\left[\left.\e^{-\lambda Z_t}\right|t\le \zeta<t+s\right]\nonumber\\
&=&\lim_{t\to+\infty}\frac{\int_0^{\phi(\lambda+\varphi(s))}\e^{-x\varphi(t+u)}\psi(\varphi(t+u))\mathrm{d}u}{\int_0^s\e^{-x\varphi(t+u)}\psi(\varphi(t+u))\mathrm{d}u}\\
&=&\lim_{t\to +\infty}
\frac{\e^{-x\varphi(t+\xi_{t,\phi(\lambda+\varphi(s))})}\psi(\varphi(t+\xi_{t,\phi(\lambda+\varphi(s))}))}
{\e^{-x\varphi(t+\xi_{t,s})}\psi(\varphi(t+\xi_{t,s}))}\dfrac{\phi(\lambda+\varphi(s))}{s},
\end{eqnarray*}
where $0<\xi_{t,\phi(\lambda+\varphi(s))}<\phi(\lambda+\varphi(s))$ and $0<\xi_{t,s}<s$.
Applying \eqref{ratio limit}, we obtain that
\[
\lim_{t\to +\infty}\frac{\psi(\varphi(t+\xi_{t,\phi(\lambda+\varphi(s))}))}{\psi(\varphi(t+\xi_{t,s}))}=1.
\]
So from Lemma \ref{lem for U}, for all $\lambda>0$,
\begin{eqnarray*}
\lim_{t\to+\infty}\pp_x\left[\left.\e^{-\lambda Z_t}\right|t\le \zeta<t+s\right]=\frac{\phi(\lambda+\varphi(s))}{s}
=\widehat{\mu}_s(\lambda).
\end{eqnarray*}
When $\alpha>0$,
by \cite[Lemma 2.1]{L}, we have for any $s\ge 0$,
\begin{equation}\label{regular varying}
\lim_{t\to+\infty}\frac{\varphi(t+s)}{\varphi(t)}=\e^{-\alpha s}.
\end{equation}
Thus taking limit in \eqref{conditioned expression}, we get
\begin{equation}\label{conditioned expression for sub}
\lim_{t\to+\infty}\pp_x\left[\left.\e^{-\lambda Z_t}\right|t\le \zeta<t+s\right]
=\lim_{t\to+\infty}\frac{\varphi(t)-\varphi(t+\phi(\lambda+\varphi({s})))}{\varphi(t)-\varphi(t+{s})}
=\frac{1-\e^{-\alpha\phi(\lambda+\varphi(s))}}{1-\e^{-\alpha s}}.
\end{equation}
By \eqref{Lap of Theta}, we have for $\lambda>0$,
$$
\int_{0+}^{+\infty}\e^{-\lambda r-\varphi(s)r}\p(\Theta\in \mathrm{d}r)=\mathrm{E}\left[\e^{-(\lambda+\varphi(s))\Theta}\right]
=1-\e^{-\alpha \phi(\lambda+\varphi(s))}.
$$
In particular,
$\mathrm{E}\left[\e^{-\varphi(s)\Theta}\right]=1-\e^{-\alpha\phi(\varphi(s))}=1-\e^{-\alpha s}$.
Consequently, we get
\[
\lim_{t\to+\infty}\pp_x\left[\left.\e^{-\lambda Z_t}\right|t\le \zeta<t+{s}\right]
=\mathrm{E}\left[\e^{-\lambda W_s}\right],
\]
where the distribution of
$W_s$ is given by \eqref{them4.2.1}.
\end{proof}

Next we shall define the distribution of $Z_{t-q}$ ($0<q<t$)
conditioned on extinction at a fixed time $t$ by taking limit of $\pp_x(Z_{t-q}\in\cdot|t\le \zeta<t+s)$ as $s\to 0+$.  Recall that
$$
\pp_x(\zeta\le t)=\e^{-x\varphi(t)}, \qquad t\ge 0.
$$
Since $\varphi'(t)=-\psi(\varphi(t))$, conditioned on $Z_0=x>0$, the distribution of $\zeta$
has a density function given by
\begin{equation}\label{density of extinction time}
f_{\zeta|Z_0}(t|x)=x\e^{-x\varphi(t)}\psi(\varphi(t)),\qquad t>0.
\end{equation}
For any $s>0$, $0<q<t$ and $\lambda>0$,
\begin{eqnarray}
&&\pp_x\left[\left.\e^{-\lambda Z_{t-q}}\right|t\leq \zeta<t+s\right]
=\frac{\pp_x\left[\e^{-\lambda Z_{t-q}}{\rm I}_{\{t\le \zeta<t+s\}}\right]}{\pp_x(t\le \zeta<t+s)}\nonumber\\
&=&\dfrac{\pp_x\left[\e^{-\lambda Z_{t-q}}\pp_{Z_{t-q}}(q\le \zeta<q+s)\right]}{\pp_x(t\leq \zeta<t+s)}\nonumber\\
&=&\dfrac{\pp_x\left[\e^{-\lambda Z_{t-q}}\int_q^{q+s}f_{\zeta|Z_0}(r|Z_{t-q})\mathrm{d}r\right]}{\int_{t}^{t+s}f_{\zeta|Z_0}(r|x)\mathrm{d}r}\nonumber\\
&\rightarrow & \dfrac{\pp_x\left[Z_{t-q}\e^{-(\lambda+\varphi(q))Z_{t-q}}\right]\psi(\varphi(q))}{x\e^{-x\varphi(t)}\psi(\varphi(t))},\label{eq4.8}
\end{eqnarray}
as $s\to 0+$.  We note that for $\lambda\ge 0$,
\begin{eqnarray}\label{Lap of derivative}
\pp_x\left[Z_{t-q}\e^{-(\lambda+\varphi(q))Z_{t-q}}\right]
&=&\left.x\e^{-x\varphi(t-q+\phi(s))}\frac{\partial}{\partial s}u_{t-q}(s)\right|_{s=\lambda+\varphi(q)}\nonumber\\
&=&x\e^{-x\varphi(t-q+\phi(\lambda+\varphi(q)))}\frac{\psi(\varphi(t-q+\phi(\lambda+\varphi(q))))}{\psi(\lambda+\varphi(q))}.
\end{eqnarray}
In particular,
\begin{equation}\label{eq4.10}
\pp_x\left[Z_{t-q}\e^{-\varphi(q)Z_{t-q}}\right]=x\e^{-x\varphi(t)}\frac{\psi(\varphi(t))}{\psi(\varphi(q))}.
\end{equation}
We can rewrite the limit in \eqref{eq4.8} as
\[
\lim_{s\to 0+}\pp_x\left[\left.\e^{-\lambda Z_{t-q}}\right|t\leq \zeta<t+s\right]=
\frac{\pp_x\left[\e^{-\lambda Z_{t-q}}\cdot Z_{t-q}\e^{-\varphi(q)Z_{t-q}}\right]}{\pp_x\left[Z_{t-q}\e^{-\varphi(q)Z_{t-q}}\right]}.
\]
The term in the right is a Laplace transform of a probability measure on $(0,+\infty)$.
For $0<q<t$, we denote this probability by
\begin{equation}\label{eq4.11}
\pp_x(Z_{t-q}\in\cdot|\zeta=t):=\lim_{s\to 0+}\pp_x\left[\left. Z_{t-q}\in\cdot\right|t\leq\zeta<t+{s}\right]
=\frac{\pp_x\left[ Z_{t-q}\e^{-\varphi(q)Z_{t-q}}; Z_{t-q}\in\cdot \right]}{\pp_x\left[Z_{t-q}\e^{-\varphi(q)Z_{t-q}}\right]}.
\end{equation}
\begin{remark}[Conditioning on extinction vs. conditioning on non-extinction]\label{rm4}
The above argument justifies the definition of the conditional law $\pp_x(Z_{t-q}\in\cdot|\zeta=t)$ for $0<q<t$ and $x>0$.
In fact, applying similar argument, one can show that the limit
$$
\pp_x\left(A|\zeta=t\right):=\lim_{s\to 0+}\pp_x\left(A|t\le \zeta< t+s\right)
$$
exists for any $x>0$,
$0<q<t$
and $A\in \mathcal{F}_{t-q}$.  On the other hand, one can also condition the CB process to be extinct at a fixed time in the sense of $h$-transforms.
Given $t>0$, let
$$
M^{(t)}_s:=Z_s\e^{-\varphi(t-s)Z_s}\psi(\varphi(t-s)),  \quad\forall 0\le s<t.
$$
It is known (cf. \cite[Lemma 4.2]{RSZ}) that $(M^{(t)}_s)_{0\leq s<t}$ is a nonnegative
$(\mathcal{F}_s)_{s<t}$-martingale.
Moreover, it is proved in \cite{RSZ} that the distribution of
$(Z_s)_{s<t}$
under the conditional probability
$\pp_x(\cdot|\zeta=t)$ is the $h$-transform of $\pp_x$ based on this martingale.  That is,
for any $0\le s<t$ and $A\in\mathcal{F}_s$,
\begin{equation}\label{def:M1}
\pp_x(A|\zeta=t)=\pp_x\Big[\frac{M^{(t)}_s}{M_0^{(t)}};A\Big].
\end{equation}
A closely related conditioning for the CB process is
conditioning the process on non-extinction.
The latter is defined by
Lambert \cite{L} in the sense of $h$-transforms.  More precisely, it is shown in \cite{L} that for any $x,t>0$ and $A\in \mathcal{F}_t$,
$$
\lim_{s\to+\infty}\pp_x\left(A|\zeta>t+s\right)=\pp^{\uparrow}_x(A),
$$
where $\pp^{\uparrow}_x$ is the $h$-transform of $\pp_x$ based on the nonnegative $(\mathcal{F}_t)$-martingale
$M_t:=Z_t\e^{\alpha t}$, that is,
\begin{equation}\label{def:M2}
\left.\frac{d\pp^{\uparrow}_x}{d\pp_x}\right|_{\mathcal{F}_t}=\frac{M_t}{M_0},\quad\forall t\ge 0.
\end{equation}
The process conditioned on non-extinction
is denoted by $Z^{\uparrow}$, and called the $Q$-process.
It is proved in \cite{L} that $Z^{\uparrow}$ is distributed as a CB process with immigration (CBI process).
In the remaining of this remark we shall show that for
any $x,t>0$ and $A\in\mathcal F_t$,
\begin{equation}\label{rm4.1}
\lim_{s\to+\infty}\pp_x(A|\zeta=t+s)=\pp^{\uparrow}_x(A).
\end{equation}
This implies that the CB process conditioned to be extinct at time $t+s$, as $s\to +\infty$, has the same law as the Q-process $Z^{\uparrow}$.
To prove \eqref{rm4.1}, we note that for any $t,x>0$ and $s>0$,
$$
\frac{M^{(t+s)}_t}{M^{(t+s)}_0}=\frac{Z_t\e^{-\varphi(s)Z_t}\psi(\varphi(s))}{Z_0\e^{-\varphi(t+s)Z_0}\psi(\varphi(t+s))}.
$$
By \eqref{ratio limit}, we have $\lim_{s\to+\infty}\psi(\varphi(s))/\psi(\varphi(t+s))=\e^{\alpha t}$. It follows that
\[
\lim_{s\to+\infty}\frac{M^{(t+s)}_t}{M^{(t+s)}_0}=\frac{Z_t\e^{\alpha t}}{x}=\frac{M_t}{x},\quad\pp_x\mbox{-a.s.}
\]
Hence by the dominated convergence theorem, we get
\[
\lim_{s\to+\infty}\pp_x(A|\zeta=t+s)=\pp_x\left(\frac{M_t}{x};A\right)=\pp^{\uparrow}_x(A).
\]
\end{remark}

In the next result we obtain the distribution of the CB process conditioned to be extinct
at a fixed time in the limit of large times.
\begin{thm}\label{thm for conditioned on fix time}
For any $q>0$, there is a positive random variable $V_q$
such that for any $\lambda,x>0$,
\begin{equation}\label{them4.3.0}
\lim_{t\to+\infty}\pp_x\left[\left.\e^{-\lambda Z_{t-q}}\right|\zeta=t\right]=
\mathrm{E}\left[\e^{-\lambda V_q}\right]=\e^{-\alpha(\phi(\lambda+\varphi(q))-q)}\frac{\psi(\varphi(q))}{\psi(\lambda+\varphi(q))}.
\end{equation}
Moreover, the distribution of $V_q$ satisfies that
\begin{equation}\label{thm4.3.1}
\p\left(V_q\in \mathrm{d}r\right)=\frac{r\mathrm{P}(W_q\in \mathrm{d}r)}{\mathrm{E}[W_q]},
\end{equation}
where $W_q$ is defined in Theorem \ref{thm:conditioned on near extinction}.
\end{thm}
\begin{proof}
Combining \eqref{eq4.8} and \eqref{Lap of derivative}, we have for all $\lambda>0$,
\begin{equation}\label{them4.3.2}
\pp_x\left[\e^{-\lambda Z_{t-q}}|\zeta=t\right]
=\e^{-x(\varphi(t-q+\phi(\lambda+\varphi(q)))-\varphi(t))}\frac{\psi(\varphi(t-q+\phi(\lambda+\varphi(q))))}{\psi(\varphi(t))}
\,\frac{\psi(\varphi(q))}{\psi(\lambda+\varphi(q))}.
\end{equation}
If $\alpha>0$, then by \eqref{regular varying} as $t\to+\infty$,
$$
\frac{\psi(\varphi(t-q+\phi(\lambda+\varphi(q))))}{\psi(\varphi(t))}\sim
\frac{\alpha\varphi(t-q+\phi(\lambda+\varphi(q)))}{\alpha\varphi(t)}
\rightarrow\e^{-\alpha(\phi(\lambda+\varphi(q))-q)}.
$$
Otherwise if $\alpha=0$, by \eqref{ratio limit}, one has
$$
\lim_{t\to+\infty}\frac{\psi(\varphi(t-q+\phi(\lambda+\varphi(q))))}{\psi(\varphi(t))}=1.
$$
In either case, one has
$$
\lim_{t\to+\infty}\frac{\psi(\varphi(t-q+\phi(\lambda+\varphi(q))))}{\psi(\varphi(t))}=\e^{-\alpha(\phi(\lambda+\varphi(q))-q)}.
$$
Hence we get \eqref{them4.3.0} by letting $t\to+\infty$ in \eqref{them4.3.2}.
It follows by the first conclusion of Theorem \ref{thm:conditioned on near extinction} that for any $\lambda>0$,
\begin{equation*}
\mathrm{E}\left[W_q\e^{-\lambda W_q}\right]=-\frac{\mathrm{d}}{\mathrm{d}\lambda}\mathrm{E}\left[\e^{-\lambda W_q}\right]
=\begin{cases}
\dfrac{\alpha}{1-\e^{-\alpha q}}\,\dfrac{1}{\psi(\lambda+\varphi(q))}\e^{-\alpha\phi(\lambda+\varphi(q))},&\quad\alpha>0;\\
\dfrac{1}{q\psi(\lambda+\varphi(q))},&\quad\alpha=0.
\end{cases}
\end{equation*}
By letting $\lambda\to 0+$, we have
$$
\mathrm{E}[W_q]
=\begin{cases}
\dfrac{\alpha}{\e^{\alpha q}-1}\,\dfrac{1}{\psi(\varphi(q))},&\quad \alpha>0;\\
\dfrac{1}{q\psi(\varphi(q))},&\quad \alpha=0.
\end{cases}
$$
Thus we get
$$
\frac{1}{\mathrm{E}[W_q]}\int_0^{+\infty}\e^{-\lambda r}r\mathrm{P}\left(W_q\in \mathrm{d}r\right)
=\frac{\mathrm{E}\left[W_q\e^{-\lambda W_q}\right]}{\mathrm{E}[W_q]}
=\e^{-\alpha(\phi(\lambda+\varphi(q))-q)}\frac{\psi(\varphi(q))}{\psi(\lambda+\varphi(q))}.
$$
This yields \eqref{thm4.3.1}.
\end{proof}

There is another way to obtain the distribution of $V_q$ for
the critical CB process
by reversing the process from the extinction time $\zeta$.
\begin{prop}
Suppose $(Z_t)_{t\geq 0}$ is a critical CB process.
For any $q>0$, under $\pp_{x}$, $Z_{\zeta-q}{\rm I}_{\{\zeta>q\}}$  converges in distribution to $V_q$ as $x\to+\infty$.
\end{prop}
\begin{proof}
For any $\lambda>0$, by the total probability formula,
\begin{eqnarray*}
\pp_x\left[\e^{-\lambda Z_{\zeta-q}}{\rm I}_{\{\zeta>q\}}\right]
=\int_q^{+\infty} f_{\eta|Z_0}(t|x)\pp_x\left[\e^{-\lambda Z_{\zeta-q}}|\zeta=t\right]\mathrm{d}t.
\end{eqnarray*}
Here $f_{\eta|Z_0}(t|x)$ is the probability density function of $\zeta$ given that $Z_0=x$.
By \eqref{density of extinction time}, \eqref{eq4.10} and \eqref{eq4.11},
we get that
\begin{eqnarray*}
&&\pp_x\left[\e^{-\lambda Z_{\zeta-q}}{\rm I}_{\{\zeta>q\}}\right]
=\psi(\varphi(q))\int_q^{+\infty} \pp_x\left[Z_{t-q}\e^{-(\lambda+\varphi(q)) Z_{t-q}}\right]\mathrm{d}t\\
&=&\psi(\varphi(q))\int_0^{+\infty} \pp_x\left[Z_t\e^{-(\lambda+\varphi(q)) Z_t}\right]\mathrm{d}t
=\psi(\varphi(q))\int_{0+}^{+\infty} y \e^{-(\lambda+\varphi(q))y}G(x,\mathrm{d}y).
\end{eqnarray*}
It follows from Theorem \ref{thm con of green} that
\begin{eqnarray*}
&&\lim_{x\to+\infty}\pp_x\left[\e^{-\lambda Z_{\zeta-q}}{\rm I}_{\{\zeta>q\}}\right]=\psi(\varphi(q))\lim_{x\to+\infty}
\int_{0+}^{+\infty} y \e^{-(\lambda+\varphi(q))y}G(x,\mathrm{d}y)\\
&=&\psi(\varphi(q))\int_{0+}^{+\infty}  y \e^{-(\lambda+\varphi(q))y}\mu(\mathrm{d}y)
=\psi(\varphi(q))\int_{0+}^{+\infty} \e^{-(\lambda+\varphi(q))y}W(y)\mathrm{d}y\\
&=&\frac{\psi(\varphi(q))}{\psi(\lambda+\varphi(q))}=\mathrm{E}[\e^{-\lambda V_q}].
\end{eqnarray*}
We observe that
\[
\lim_{x\to+\infty}\pp_x(\zeta\leq q)=\lim_{x\to+\infty}1-\e^{-x\varphi(q)}=0.
\]
Thus for every $\lambda>0$,
$$\pp_{x}\left[\e^{-\lambda Z_{\zeta-q}{\rm I}_{\{\zeta>q\}}}\right]=\pp_{x}\left[\e^{-\lambda Z_{\zeta-q}}{\rm I}_{\{\zeta>q\}}\right]+\pp_{x}\left(\zeta\le q\right)\to \mathrm{E}\left[\e^{-\lambda V_{q}}\right]\mbox{ as }x\to+\infty.$$
We complete the proof.
\end{proof}

Finally we give some examples
to illustrate the results obtained in this subsection.
\begin{exa}
Suppose $(Z_t)_{t\ge 0}$ is a critical CB process with branching mechanism $\psi(\lambda)=\lambda^{\beta}$ ($1<\beta\le 2$). Then the corresponding scale function $W(x)=x^{\beta-1}/\Gamma(\beta)$ for $x>0$, and $\varphi(t)=((\beta-1)t)^{-1/(\beta-1)}$ for $t>0$. So the stationary measure on $(0,+\infty)$ is given by $\mu(\mathrm{d}x)=\frac{x^{\beta-2}}{\Gamma(\beta)}\mathrm{d}x$ for $x>0$.

By Theorem \ref{thm:conditioned on near extinction}, for any $q>0$, conditioned on $\{t-q\le \zeta<t\}$, $Z_{t-q}$ converges in distribution to a positive random variable $W_{q}$ as $t\to+\infty$,
where $W_{q}$ has a Gamma$\left([q(\beta-1)]^{-1/(\beta-1)},\beta-1\right)$-distribution with the probability density function given by
 \[
 g_q(x)=\frac{x^{\beta-2}}{q\Gamma(\beta)}\exp\left\{-\frac{x}{[q(\beta-1)]^{1/(\beta-1)}}\right\},\qquad x>0.
 \]
By Theorem \ref{thm for conditioned on fix time}, for any $q>0$, conditioned on $\{\zeta=t\}$, $Z_{t-q}$ converges in distribution to a positive random
variable $V_q$ as $t\to+\infty$,
where $V_q$ has a Gamma$\left([q(\beta-1)]^{-1/(\beta-1)},\beta\right)$-distribution with the probability density function given by
\[
 p_q(x)=\frac{x^{\beta-1}}{\Gamma(\beta)[q(\beta-1)]^{\frac{\beta}{\beta-1}}}\exp\left\{-\dfrac{x}{[q(\beta-1)]^{1/(\beta-1)}}\right\},\qquad x>0.
 \]
In particular, when $\beta=2$,  $W_q$ is distributed according to the exponential distribution with parameter $1/q$, and
$V_q$ is distributed according to
Gamma distribution with parameter $(1/q,2)$.
\end{exa}
\begin{exa}
Suppose $(Z_t)_{t\ge 0}$ is a subcritical CB process with branching mechanism $\psi(\lambda)=\lambda+\lambda^{2}$. Then by elementary calculation, one gets that $W(x)=1-\e^{-x}$ for $x>0$, $\phi(\lambda)=\ln(1+\lambda^{-1})$ for $\lambda>0$ and $\varphi(t)=(\e^{t}-1)^{-1}$ for $t>0$. The Laplace transform of the Yaglom distribution $\nu_{1}(\mathrm{d}x)$ is given by
$$
\widehat{\nu}_1(\lambda)=1-\e^{-\phi(\lambda)}=\frac{1}{\lambda+1},\quad\forall \lambda>0.
$$
So the corresponding Yaglom distribution is the exponential distribution with parameter $1$.
It follows by Theorem \ref{thm:conditioned on near extinction} that for any $q>0$, conditioned on $\{t-q\le \zeta<t\}$, $Z_{t-q}$ converges in distribution to a positive random variable $W_{q}$ as $t\to+\infty$, where $W_{q}$ is exponentially distributed with parameter $1+(\e^{q}-1)^{-1}$.
Moreover by Theorem \ref{thm for conditioned on fix time}, for any $q>0$, conditioned on $\{\zeta=t\}$, $Z_{t-q}$ converges in distribution to a positive random variable $V_q$ as $t\to+\infty$, where $V_{q}$ is distributed according to Gamma distribution with parameter $(1+(\e^q-1)^{-1},2)$.
\end{exa}
\subsection{Further properties of the limiting distributions}
In this subsection we will investigate properties of the distribution of
$V_{q}$ obtained in Theorem \ref{thm for conditioned on fix time}.
We show that it is infinitely divisible, and give a representation of its L\'{e}vy-Khintchine triplet. Then we show the distribution of $V_q$
is weakly convergent as $q\to+\infty$, and give a necessary and sufficient condition for the limit distribution to be non-degenerate.

Recall that $(X_t)_{t\ge 0}$ is a spectrally positive L\'{e}vy process with Laplace exponent $\psi$ and $W$ is a corresponding scale function.
Under the assumption \eqref{extinct assumption}, $X$ has unbounded variation. Hence by \cite[Lemma 3.1]{KKR} $W(0)=0$.
Moreover, by \cite[Lemma 8.2]{Kyprianou} (and the reference therein), the restriction of $W$ to $(0,+\infty)$ is continuously differentiable.
\begin{prop}\label{prop4.5}
For any $q>0$, the distribution of $V_q$ is infinitely divisible and its Laplace exponent
$l_q(\lambda):=-\ln\mathrm{E}\left[\e^{-\lambda V_q}\right]$
is given by
\begin{equation}\label{prop4.5.0}
l_q(\lambda)=\int_{\varphi(q)}^{\lambda+\varphi(q)}\frac{\psi^{'}(s)-\alpha}{\psi(s)}\mathrm{d}s,\qquad\lambda>0.
\end{equation}
Moreover,
$l_q(\lambda)$
has the L\'{e}vy-Khintchine decomposition
\begin{equation}\label{prop4.5.1}
l_q(\lambda)=b_q\lambda+\int_0^{+\infty}\left(1-\e^{-\lambda x}\right)\frac{v_q(x)}{x}\mathrm{d}x,
\end{equation}
where $b_{q}=0$,
\begin{equation}\label{prop4.5.2}
 v_q(x)=\e^{-\varphi(q)x}\left[\sigma^2 W'(x)
 +\int_{(0,+\infty)}\left(W(x)-W(x-r)\right)r\pi(\mathrm{d}r)\right],\qquad x>0,
\end{equation}
and $W'(x)$ denotes the derivative of $W(x)$.
\end{prop}

\begin{proof}
By Theorem \ref{thm for conditioned on fix time} we have
$$
l_q(\lambda)=\ln\frac{\psi(\lambda+\varphi(q))}{\psi(\varphi(q))}+\alpha\left(\phi(\lambda+\varphi(q))-q\right).
$$
Consequently,
$$
l'_q(\lambda)=\frac{\psi'(\lambda+\varphi(q))-\alpha}{\psi(\lambda+\varphi(q))},\quad\forall \lambda>0.
$$
Thus \eqref{prop4.5.0}
follows by taking integrals on both sides of the above equation. Note that $l_q(\lambda)\to 0$ as $\lambda\to 0+$.
So to show the distribution of $V_q$ is infinitely divisible, it suffices to show that $l_q(\lambda)$ is a Berstein function,
or equivalently, the first derivative of $l_a(\lambda)$ is completely monotone, that is, $(-1)^nl^{(n+1)}_{q}(\lambda)\ge 0$ for all $\lambda>0$ and $n=0,1,2,...$.

We note that
$$
\psi'(u)-\alpha=\sigma^2 u+\int_{(0,+\infty)}\left(1-\e^{-ur}\right)r\pi(\mathrm{d}r),\quad\forall u>0,
$$
is the Laplace exponent of a L\'evy subordinator.
Applying \cite[eq. (3.15) and eq. (3.16)]{KR-M} by taking $F(u)=\psi'(u)-\alpha$ and $R(u)=-\psi(u)$ (and correspondingly $b=\sigma^2$ and
$m(\mathrm{d}r)=r\pi(\mathrm{d}r)$),
we get that
\begin{eqnarray*}
\frac{F(u)}{\psi(u)}&=&\sigma^2 W(0)+\sigma^2\int_0^{+\infty}\e^{-u x}W'(x)\mathrm{d}x\nonumber\\
&&\quad+\int_0^{+\infty}\e^{-u x}\left[\int_{(0,+\infty)}\left(W(x)-W(x-r)\right)r\pi(\mathrm{d}r)\right]\mathrm{d}x,\quad\forall u>0.
\end{eqnarray*}
It follows that for $\lambda>0$,
\begin{eqnarray}
l'_q(\lambda)&=&\frac{F(\lambda+\varphi(q))}{\psi(\lambda+\varphi(q))}\nonumber\\
&=&\sigma^2W(0)+\sigma^2\int_0^{+\infty}\e^{-\lambda
x}\left(\e^{-\varphi(q)x}W'(x)\right)\mathrm{d}x\nonumber\\
&&+\int_0^{+\infty}\e^{-\lambda x}\left[\e^{-\varphi(q)x}\int_{(0,+\infty)}\left(W(x)-W(x-r)\right)r\pi(\mathrm{d}r)\right]\mathrm{d}x.\label{prop4.5.3}
\end{eqnarray}
One can easily show by the above identity that $l'_q(\lambda)$ is completely monotone.
Suppose the L\'evy-Khintchine decomposition of $l_q(\lambda)$ is given by
$$
l_q(\lambda)=b_q\lambda+\int_{(0,+\infty)}\left(1-\e^{-\lambda x}\right)\Gamma_q(\mathrm{d}x)
,\quad \lambda>0,
$$
where $b_q\ge 0$ and $\Gamma_q$ is a measure on $(0,+\infty)$ such that $\int_{(0,+\infty)}(1\wedge x)\Gamma_q(\mathrm{d}x)<+\infty$. Then
$$
l'_q(\lambda)=b_q+\int_{(0,+\infty)}\e^{-\lambda x}x\Gamma_q(\mathrm{d}x).
$$
Comparing the right hand side with that of \eqref{prop4.5.3}, we deduce that $b_q=\sigma^2W(0)=0$
and $\Gamma_q(\mathrm{d}x)=\frac{v_q(x)}{x}\mathrm{d}x$
with $v_q(x)$ being given by \eqref{prop4.5.2}.
\end{proof}

\begin{prop}\label{prop V_infty}
If
\begin{equation}\label{con1}
\alpha>0\mbox{ and }\int^{+\infty}r\ln r\pi(\mathrm{d}r)<+\infty,
\end{equation}
then $V_q$ converges in distribution as $q\to+\infty$ to a positive random variable $V_\infty$.
The distribution of $V_\infty$ has the following properties.
\begin{itemize}
\item[(i)]It is of the size-biased Yaglom distribution
$$
\p(V_\infty \in \mathrm{d}r)=\frac{r\p(\Theta\in \mathrm{d}r)}{\mathrm{E}[\Theta]}.
$$
\item[(ii)] It is infinitely divisible.
\item[(iii)] Its Laplace exponent
$
l_\infty(\lambda):=-\ln\mathrm{E}\left[\e^{-\lambda V_\infty}\right]
$
is given by
$$
l_{\infty}(\lambda)=\int_0^{\lambda}\frac{\psi'(s)-\alpha}{\psi(s)}\mathrm{d}s,
\quad \lambda>0.
$$
\item[(iv)] $l_\infty(\lambda)$ has the L\'evy-Khintchine decomposition
$$
l_\infty(\lambda)=b_\infty\lambda+\int_0^{+\infty}\left(1-\e^{-\lambda x}\right)\frac{v_\infty(x)}{x}\mathrm{d}x,
$$
where $b_\infty=0$,
and
\begin{equation*}
 v_\infty(x)=\sigma^2W'(x)+\int_0^{+\infty}\left(W(x)-W(x-r)\right)r\pi(\mathrm{d}r),\quad x>0.
\end{equation*}
\end{itemize}
Otherwise if \eqref{con1} fails, then $V_q$ converges in probability as $q\to+\infty$ to infinity.
\end{prop}
\begin{proof}
First we claim that \eqref{con1} holds if and only if
$$
\int_{0+}\frac{\psi'(s)-\alpha}{\psi(s)}\mathrm{d}s<+\infty.
$$
In fact, if $\alpha=0$, then
$\int_{0+}\psi'(s)/\psi(s)\mathrm{d}s=\int_{0+}d \ln \psi(s)=+\infty$.  On the other hand, if $\alpha>0$, we have
$$
\frac{s\psi'(s)}{\psi(s)}=\frac{\alpha+\sigma^2s+\int_{(0,+\infty)}\left(1-\e^{-sr}\right)r\pi(\mathrm{d}r)}{\alpha+\frac{1}{2}\sigma^2s
+\int_{(0,+\infty)}\left(\frac{\e^{-s r}-1+s r}{sr}\right)r\pi(\mathrm{d}r)}\to 1\mbox{ as }s\to 0+.
$$
Hence $\psi'(s)/\psi(s)\sim 1/s$ as $s\to 0+$. This implies further that
$$
\int_{0+}\frac{\psi'(s)}{\psi(s)}-\frac{\alpha}{\psi(s)}\mathrm{d}s<+\infty\mbox{ iff }\int_{0+}\frac{1}{s}-\frac{\alpha}{\psi(s)}\mathrm{d}s<+\infty.
$$
By \cite[Lemma 2.1]{L}, the latter holds iff \eqref{con1} holds. Hence we prove the claim.

Let $l_q(\lambda)$ be the Laplace exponent of $V_q$. It follows by \eqref{prop4.5.0}
and the above claim that
$$
\lim_{q\to+\infty}l_q(\lambda)=\lim_{q\to+\infty}\int_{\varphi(q)}^{\lambda+\varphi(q)}\frac{\psi'(s)-\alpha}{\psi(s)}\mathrm{d}s
=\begin{cases}
\int_{0}^{\lambda}\frac{\psi'(s)-\alpha}{\psi(s)}\mathrm{d}s&\mbox {if \eqref{con1} holds;}\\
+\infty&\mbox{ otherwise.}
\end{cases}
$$
So $V_q$ converges in distribution as $q\to+\infty$ to some random variable $V_\infty$ if \eqref{con1} holds,
and $V_q$ converges in probability to infinity if \eqref{con1} fails.
When \eqref{con1} holds, it follows by \eqref{thm4.3.1} and \eqref{them4.2.1} that
$$
\p(V_q\in \mathrm{d}r)=\frac{r\e^{-\varphi(q)r}\p(\Theta\in \mathrm{d}r)}{\mathrm{E}[\Theta \e^{-\varphi(q)\Theta}]}.
$$
Hence (i) follows by letting $q\to+\infty$.  The statements (ii)-(iv)
follow directly from Proposition \ref{prop4.5}.
\end{proof}

Recall that
$(Z^{\uparrow}_t)_{t\ge 0}$
 is the $Q$-process defined in Remark \ref{rm4}.
The next result shows that $Z^{\uparrow}_t$ converges in distribution as $t\to+\infty$,
and its limit distribution is equal to that of $V_q$ as $q\to+\infty$.  Since
$(Z^{\uparrow}_t)_{t\ge 0}$ is a CBI process, criteria for convergence in distribution and properties of the limiting distribution can
readily be found in \cite{KR-M}, but since they follow very easily from Theorem \ref{thm for conditioned on fix time} and
then Proposition \ref{prop V_infty}, we present the proof here for the sake of being more self-contained.

\begin{prop}\label{prop1}
If \eqref{con1} holds, then $Z^{\uparrow}_t$ converges in distribution as $t\to+\infty$ to a positive random variable $Z^{\uparrow}_{\infty}$ which is equal in distribution to $V_\infty$ defined in Proposition \ref{prop V_infty}.
Otherwise if \eqref{con1} fails, $Z^{\uparrow}_t$ converges in probability as $t\to+\infty$ to infinity.
\end{prop}
\begin{proof}
Fix an arbitrary $x>0$. We shall prove the following: For all $\lambda>0$,
\begin{equation}\label{prop1.1}
\lim_{t\to+\infty}\pp^{\uparrow}_x\left[\e^{-\lambda Z^{\uparrow}_t}\right]
=\begin{cases}
\mathrm{E}[\e^{-\lambda V_\infty}], &\quad\mbox{if \eqref{con1} holds;}\\
0,& \quad\mbox{otherwise}.
\end{cases}
\end{equation}
Fix $\lambda>0$. Suppose $s>0$ is sufficiently large such that $\varphi(s)<\lambda$.
Suppose $t\in (s,+\infty)$. Recall the definitions of the martingales $(M^{(t)}_r)_{0\le r<t}$ and $(M_r)_{r\ge 0}$
given in \eqref{def:M1} and \eqref{def:M2} respectively.
 It is easy to see that for $t>s$,
\begin{equation*}
\frac{M_{t-s}}{M_0}=\frac{\psi(\varphi(t))}{\psi(\varphi(s))}\e^{\alpha(t-s)+\varphi(s)Z_{t-s}-\varphi(t)x}
\frac{M^{(t)}_{t-s}}{M^{(t)}_0},\quad\pp_x\mbox{-a.s.}
\end{equation*}
Thus we have for $t>s$,
\begin{eqnarray}\label{prop1.2}
\pp^{\uparrow}_x\left[\e^{-\lambda Z^{\uparrow}_{t-s}}\right]&=&\pp_x\left[\frac{M_{t-s}}{M_0}\e^{-\lambda Z_{t-s}}\right]\nonumber\\
&=&\frac{\psi(\varphi(t))}{\psi(\varphi(s))}\e^{\alpha(t-s)-\varphi(t)x}\pp_x\left[\frac{M^{(t)}_{t-s}}{M^{(t)}_0}\e^{-(\lambda-\varphi(s))Z_{t-s}}\right]\nonumber\\
&=&\mathrm{I}(\alpha,t,s)\times \mathrm{II}(\lambda,t,s),
\end{eqnarray}
where
$$
\mathrm{I}(\alpha,t,s):=\frac{\psi(\varphi(t))}{\psi(\varphi(s))}\e^{\alpha(t-s)-\varphi(t)x}\,\mbox{ and}\quad \mathrm{II}(\lambda,t,s)
:=\pp_x\left[\e^{-(\lambda-\varphi(s))Z_{t-s}}|\zeta=t\right].
$$
If $\alpha=0$, we have
\begin{equation}\label{prop1.2.1}
\lim_{r\to 0+}\psi(r)\e^{\alpha \phi(r)}=\lim_{r\to 0+}\psi(r)=0.
\end{equation}
Otherwise if $\alpha>0$, we note that by \eqref{Lap of Theta}
$$
\mathrm{E}[\Theta \e^{-r\Theta}]=-\alpha\phi'(r)\e^{-\alpha \phi(r)}=\frac{\alpha}{\psi(r)\e^{\alpha\phi(r)}}\quad\forall r>0.
$$
Consequently, we have
\begin{equation}\label{prop1.2.2}
\lim_{r\to 0+}\psi(r)\e^{\alpha\phi(r)}=\lim_{r\to 0+}\frac{\alpha}{\mathrm{E}\left[\Theta\e^{-r\Theta}\right]}
=\begin{cases}
\frac{\alpha}{\mathrm{E}[\Theta]}, &\mbox{ if \eqref{con1} holds;}\\
0,&\mbox{ if $\alpha>0$ and $\int^{+\infty}r\log r\pi(\mathrm{d}r)=+\infty$.}
\end{cases}
\end{equation}
Combing \eqref{prop1.2.1}, \eqref{prop1.2.2} with the fact that $\lim_{t\to+\infty}\varphi(t)=0$, we get that
\begin{equation}\nonumber
\lim_{t\to +\infty}\psi(\varphi(t))\e^{\alpha t}
=\lim_{t\to +\infty}\psi(\varphi(t))\e^{\alpha\phi(\varphi(t))}
=\begin{cases}
\frac{\alpha}{\mathrm{E}[\Theta]}, &\mbox{ if \eqref{con1} holds;}\\
0,&\mbox{ otherwise.}
\end{cases}
\end{equation}
It follows that
\begin{equation}\label{prop1.3}
\lim_{t\to+\infty}\mathrm{I}(\alpha,t,s)=
\begin{cases}
\frac{\alpha}{\mathrm{E}[\Theta]\psi(\varphi(s))}\e^{-\alpha s},&\quad\mbox{if \eqref{con1} holds;}\\
0,& \quad\mbox{otherwise}.
\end{cases}
\end{equation}
On the other hand, by Theorem \ref{thm for conditioned on fix time},
\begin{equation}\label{prop1.4}
\lim_{t\to+\infty}\mathrm{II}(\lambda,t,s)=\mathrm{E}\left[\e^{-(\lambda-\varphi(s))V_{s}}\right].
\end{equation}
Combining \eqref{prop1.2}, \eqref{prop1.3} and \eqref{prop1.4}, we have
\begin{equation*}
\lim_{t\to+\infty}\pp^{\uparrow}_x\left[\e^{-\lambda Z^{\uparrow}_t}\right]=
\begin{cases}
\frac{\alpha}{\mathrm{E}[\Theta]\psi(\varphi(s))}\e^{-\alpha s}\mathrm{E}\left[\e^{-(\lambda-\varphi(s))V_s}\right],&\quad\mbox{if \eqref{con1} holds;}\\
0,& \quad\mbox{otherwise}.
\end{cases}
\end{equation*}
Hence \eqref{prop1.1} follows by letting $s\to+\infty$, and we prove the first assertion.
If \eqref{con1} fails, we have
$\lim_{t\to+\infty}\pp_x^{\uparrow}\left[\e^{-\lambda Z^{\uparrow}_t}\right]=0$ for all $\lambda>0$.
Thus for any $M>0$,
$$
\pp^{\uparrow}_x\left(Z^{\uparrow}_t\le M\right)
=\pp^{\uparrow}_x\left(\e^{-Z^{\uparrow}_t}\ge \e^{-M}\right)\le \e^M\pp^{\uparrow}_x\left[\e^{-Z^{\uparrow}_t}\right]\to 0
$$
as $t\to +\infty$. Consequently
$\lim_{t\to+\infty}\pp^{\uparrow}_x\left(Z^{\uparrow}_{t}>M\right)=1$ for all $M>0$, and so $Z^{\uparrow}_t$ converges to infinity in probability.
 Hence we prove the second assertion.
\end{proof}

One can see from Proposition \ref{prop1} and Theorem \ref{thm for conditioned on fix time} that,
the two double limits coincide:
$$\lim_{s\to+\infty}\lim_{t\to +\infty}\pp_x(Z_t\in A|\zeta=t+s)=\lim_{t\to+\infty}\lim_{s\to +\infty}\pp_x( Z_t\in A|\zeta=t+s),
$$
for any Borel set $A\subset (0,+\infty)$
with $\p(V_\infty\in\partial A)=0$, and any $x>0$.
Moreover, the limit is non-degenerate if and only if \eqref{con1} holds.


\bigskip

\textbf{Acknowledgment}:
The research of Rongli Liu is supported by NSFC (Grant No. 12271374).
The research of Yan-Xia Ren is supported by
NSFC (Grant Nos. 12071011 and  12231002) and the Fundamental Research Funds for Central Universities, Peking University LMEQF.
The research of Ting Yang is supported by NSFC (Grant Nos. 12271374 and 12371143).

\appendix

\section{Appendix}

\begin{lem}\label{A1}
\begin{enumerate}
\item   Suppose $\nu_n,\nu$ are finite measures on $(0,+\infty)$ with $\widehat{\nu}(0)>0$. Then $\nu_n$ converges weakly to $\nu$ if for all $\lambda\ge 0$,
    \begin{equation}\label{A1.1}
    \widehat{\nu}(\lambda)<+\infty\mbox{ and }\,\,
    \widehat{\nu_n}(\lambda)\to\widehat{\nu}(\lambda)\mbox{ as }n\to+\infty.
    \end{equation}
\item  Suppose $\nu_n,\nu$ are measures
on $(0,+\infty)$ with $0<\widehat{\nu}(\beta)<+\infty$ for some $\beta>0$.
Then $\nu_n$ converges vaguely to $\nu$ if \eqref{A1.1} holds for all $\lambda\geq \beta$.
\end{enumerate}
\end{lem}
\begin{proof}
 (1) Without loss of generality we assume $\widehat{\nu_n}(0)>0$ for every $n\ge 1$.
 Let $\rho_n(\cdot):=\frac{\nu_n(\cdot)}{\widehat{\nu_n}(0)}$
 and $\rho(\cdot):=\frac{\nu(\cdot)}{\widehat{\nu}(0)}$.
 Then $\rho_n$ and $\rho$ are probability measures on $(0,+\infty)$ with
$\widehat{\rho_n}(\lambda)=\widehat{\nu_n}(\lambda)/\widehat{\nu_n}(0)$
and $\widehat{\rho}(\lambda)=\widehat{\nu}(\lambda)/\widehat{\nu}(0)$ for all $\lambda\ge 0$.
\eqref{A1.1} implies that $\rho_n$ converges weakly to $\rho$. The weak convergence of $\nu_n$ follows
from the weak convergence of $\rho_n$ immediately.

(2) Since $\nu_{n}$ and $\nu$ can be viewed as measures on $[0,+\infty)$ by setting $\nu_{n}(\{0\})=\nu(\{0\})=0$, this assertion is a direct result of \cite[Theorem 8.5.a]{BW}.
\end{proof}





%
%
%
%

\medskip

{\bf Rongli Liu}

School of Mathematics and Statistics, Beijing jiaotong University,

Beijing, 100044, P. R. China.

E-mail: rlliu@bjtu.edu.cn

\medskip

{\bf Yan-Xia Ren}

LMAM School of Mathematical Sciences \& Center for Statistical Science,

Peking University,

Beijing, 100871, P.R. China.

E-mail: yxren@math.pku.edu.cn

\medskip
{\bf Ting Yang}

School of Mathematics and Statistics, Beijing Institute of Technology,

Beijing, 100081, P.R.China;

Beijing Key Laboratory on MCAACI,

Beijing, 100081, P.R. China.

Email: yangt@bit.edu.cn


\begin{thebibliography}{99}
\bibliographystyle{APT}
\footnotesize
\bibitem{AD} {\sc Abraham, R. and Delmas, J.~F.} (2009).
Williams' decomposition of the L\'{e}vy continuum random tree and simultaneous extinction probability for populations with neutral mutations. {\em Stochastic Process. Appl.} {\bf 119}, 1124--1143.
\bibitem{AR} {\sc Alsmeyer, G. and R\"{o}sler, U.} (2006). The Martin entrance boundary of the Galton-Watson process.
{\em Ann. I. H. Poincare-P.R.} {\bf 42(5)}, 591--606.
\bibitem{AH} {\sc Asmussen, S. and Hering, H.} (1983). {\em Branching Processes}, Birkh\"{a}user, Boston.
\bibitem{AN} {\sc Athreya, K.~B. and Ney, P.~E.} (1972). {\em Branching Processes}, Die Grundlehren der mathematischen Wissenschaften, Band 196, Springer-Verlag, New York-Heidelberg.
\bibitem{Bertoin} {\sc Bertion, J.} (1998).  {\em L\'{e}vy Processes}, Cambridge University Press, Cambridge.
\bibitem{BW} {\sc Bhattacharya, R. and Waymire, E.~C.} (2007). {\em A Basic Course in Probability Theory}, 1st ed., Springer-Verlag, New York.
\bibitem{Bin} {\sc Bingham, N.~H.} (1976).  Continuous branching processes and spectral positivity. {\em Stochastic Process. Appl.} {\bf 4}, 217--242.
\bibitem{BGT} {\sc Bingham, N.~H., Goldie C.~M. and Teugels, J.~L.} (1987). {\em Regular Variation}, Cambridge University Press, Cambridge.
\bibitem{CLP} {\sc Chazal, M., Loeffen, R. and Patie, P.} (2015). Smoothness of continuous state branching with immigration semigroups.
 {\em J. Math. Anal. Appl.} {\bf 459(2)}, 619--660.
\bibitem{DH} {\sc Delmas, J.~F. and H\'enard, O.} (2013). A Williams decomposition for spatially dependent super-processes. {\em Electron. J. Probab.} {\bf 18(37)}, 1--43.
\bibitem{E} {\sc Esty, W.~W.} (1976). Diffusion limits of critical branching processes conditioned on extinction in the near future. {\em J. Appl. Prob.} {\bf 13}, 247--254.
\bibitem{Grey} {\sc Grey, D.~R.} (1974). Asymptotic behaviour of continuous time continuous state space branching processes. {\em J. Appl. Probab.} {\bf 11}, 669--677.
\bibitem{Ha} {\sc Harris, T.~E.} (1963). {\em The Theory of Branching Processes}. Springer, Berlin-G\"{o}ttinger-Heiderberg.
\bibitem{H} {\sc Hoppe, F.~M.} (1977). Representations of invariant measures on multitype Galton-Watson processes. {\em Ann. Probab}.
 {\bf 5(2)}, 291--297.
\bibitem{I} {\sc Imomov, A.~A.} (2014). Limit properties of transition functions of continuous-time Markov branching processes.
 {\em Int. J. Stoch. Anal.} Art. ID 409345, 10 pp.
\bibitem{Kallenberg} {\sc Kallenberg, O.} (2017). {\em Random Measures, Theory and Applications}. Springer, Cham.

\bibitem{KR-M} {\sc Keller-Ressel, M. and Mijatovi\'{c}, A.} (2012).  On the limit distributions of continuous-state branching processes with immigration. {\em Stochastic Process. Appl.} {\bf 122}, 2329--2345.
\bibitem{KKR} {\sc Kuznetsov, A., Kyprianov, A.~E. and Rivero, V.} (2012).  The theory of scale functions for spectrally negative L\'{e}vy processes. In: {\em L\'{e}vy Matters II}. Lecture Notes in Mathematics. {\bf 2061}. Springer, Berlin, Heidelberg, pp. 97--186.
\bibitem{Kyprianou} {\sc Kyprianou, A.~E.} (2014). {\em Fluctuations of L\'{e}vy Processes with Applications}. 2nd~edn. Universitext. Springer, Heidelberg.
\bibitem{L} {\sc Lambert, A.} (2007). Quasi-stationary distributions and the continuous-state branching process conditioned to be never
extinct. {\em Electron. J. Probab.} {\bf 12(14)}, 420--446.
\bibitem{Li2} {\sc Li, Z.} (2000) Asymptotic behaviour of continuous time and state branching processes. {\em J. Austral. Math. Soc. (Series A)}. {\bf 68}, 68--84.
\bibitem{Li}  {\sc Li, Z.} (2011). {\em Measure-Valued Branching Markov Processes}. Probability and its Applications (New York). Springer, Heidelberg.
\bibitem{M} {\sc Maillard, P.} (2018). The $\lambda$-invariant measures of subcritical Bienaym\'{e}-Galton-Watson processes.
{\em Bernoulli} {\bf 24(1)}, 297--315.
\bibitem{NT} {\sc Nakagawa, T.} (1984). On the reverse process of a critical multitype Galton-Watson
process without variances. {\em J. Multivariate Anal}. {\bf 14}, 94--100.
\bibitem{Y} {\sc Ogura, Y.} (1969). Spectral representation for branching processes on the real half line.
{\em Publ. Res. Inst. Math. Sci.} {\bf 5}, 423--441.
\bibitem{Y1} {\sc Ogura, Y.} (1970). Spectral representation for branching processes with immigration on the real half line.
{\em Publ. Res. Inst. Math. Sci.} {\bf 6}, 307--321.
\bibitem{Y2} {\sc Ogura, Y.} (1974). Spectral representation for continuous state branching processes.
{\em Publ. Res. Inst. Math. Sci.} {\bf 10}, 51--75.
\bibitem{OS} {\sc Ogura, Y. and Shiotani, K.} (1976). On invariant measures of critical multitype Galton-Watson processes. {\em Osaka J. Math.} {\bf 13}, 83--98.
\bibitem{AGP} {\sc Pakes, A.~G.} (1999). Revisiting conditional limit theorems for the mortal simple branching process. {\em Bernoulli} {\bf 5(6)}, 969--998.
\bibitem{Pakes} {\sc Pakes, A.~G.} (2008). Conditional limit theorems for continuous time and state branching process. In: {\em Records and Branching Processes}, ed. M. Ahsanullah and G. P. Yanev, Nova Science Publishers, pp. 63--103.
\bibitem{P} {\sc Papangelou, F.} (1968). A lemma on the Galton-Watson process and some of its consequences. {\em Proc. Amer. Math. Soc.} {\bf 19}, 1469--1479.
\bibitem{RSZ} {\sc Ren, Y.-X., Song, R. and Zhang, R.} (2018). Williams decomposition for superprocesses. {\em Electron. J. Probab.} {\bf 23(23)}, 1--33.
\bibitem{RYZ} {\sc Ren, Y.-X., Yang, T. and Zhao, G.-H.} (2014). Conditional limit theorems for critical continuous-state branching processes.
{\em Sci. China Math.} {\bf 57(12)}, 2577--2588.
\end{thebibliography}
\end{document}